\documentclass[11pt,a4paper]{amsart}
\usepackage{mathtools}
\usepackage{amssymb}
\usepackage{enumitem}
\setlist[enumerate]{label=\emph{(\roman*)}}
\usepackage[T1]{fontenc}
\usepackage{fullpage}
\usepackage{color}
\usepackage[colorlinks=true]{hyperref}
\hypersetup{urlcolor=blue, citecolor=cyan}
\hypersetup{citecolor=blue}
\mathtoolsset{showonlyrefs}
\newtheorem{theorem}{Theorem}[]

\newtheorem{lemma}[theorem]{Lemma}
\newtheorem{proposition}[theorem]{Proposition}

\theoremstyle{definition}

\newtheorem*{problem*}{Problem}

\newtheorem{remark}[theorem]{Remark}

\numberwithin{equation}{section}
\def\MMA{\textnormal{M}}
\def\MMB{\textnormal{N}}
\def\GG{\textnormal{G}}
\def\YY{Y}

\def\bN {\mathbb{N}}

\def\bR {\mathbb{R}}
\def\bS {\mathbb{S}}

\def\WW {\vec W}
\def\bWW {\WW_{\bGa}}
\def\bWWt {\WW_{\bGa(t)}}

\def\bW {W_{\bGa}}
\def\zz {z}
\def\bzz {\bs z}
\def\bb {\bs b}
\def\bla {\bs \lambda}
\def\bGa {\bs \Gamma}
\def\ini {T}
\def\yy {\bs y}

\def\cB {\mathcal{B}}
\def\cC {\mathcal{C}}

\def\cE {\mathcal{E}}
\def\cF {\mathcal{F}}
\def\cG {\mathcal{G}}
\def\cH {\mathcal{H}}
\def\cI {\mathcal{I}}
\def\cJ {\mathcal{J}}
\def\cK {\mathcal{K}}

\def\cM {\mathcal{M}}

\def\cY {Y}
\def\YY {Y}

\def\gv {\vec g}
\def\uv {\vec u}

\def\grad {{\nabla}}

\def\la {\langle}
\def\ra {\rangle}

\newcommand{\tx}[1]{\mbox{#1}}

\newcommand{\wto}{\rightharpoonup}

\newcommand{\wt}[1]{\widetilde{#1}}
\newcommand{\bs}[1]{\boldsymbol{#1}}

\newcommand{\bcc}{\boldsymbol{c}}
\newcommand{\cc}{c}

\newcommand{\spn}{\operatorname{span}}

\newcommand{\dist}{\operatorname{dist}}

\newcommand{\Id}{\operatorname{Id}}

\DeclareMathOperator{\diag}{diag}

\newcommand{\ud}{\textnormal{\,d}}
\newcommand{\vd}{\textnormal{d}}

\newcommand{\vD}{\textnormal{D}}
\newcommand{\dd}[1]{{\frac{\vd}{\vd{#1}}}}

\newcommand{\uln}[1]{{\underline{ #1 }}}

\title[Multi-bubble for the $5$D energy-critical wave equation]{Construction of multi-bubble solutions for the energy-critical wave equation in dimension $5$}
\author[J.~Jendrej]{Jacek Jendrej}
\address{CNRS and Universit\'e Paris 13, LAGA, UMR 7539, 99 av J.-B.~Cl\'ement, 93430 Villetaneuse, France}
\email{jendrej@math.univ-paris13.fr}
\author[Y.~Martel]{Yvan Martel}
\address{CMLS, \'Ecole Polytechnique, CNRS, Institut Polytechnique de Paris, 91128 Palaiseau, France}
\email{yvan.martel@polytechnique.edu}

\begin{document}
\begin{abstract}
We prove the existence of a global solution of the energy-critical focusing wave equation in dimension $5$ blowing up in infinite time at any $K$ given points $z_k$ of~$\bR^5$, where $K\geq 2$. The concentration rate of each bubble is asymptotic to $c_k t^{-2}$ as $t\to \infty$, where the~$c_k$ are positive constants depending on the distances between the blow-up points $z_k$.
This result complements previous constructions of blow-up solutions and multi-solitons of the energy-critical wave equation in various dimensions $N\geq 3$.
 \end{abstract}
\maketitle
\section{Introduction}
\label{sec:intro}
\subsection{Main result}
\label{ssec:setting}

We consider the energy-critical focusing wave equation in dimension~$5$
\begin{equation}
\label{eq:nlw}
\partial_t^2 u(t, x) = \Delta u(t, x) + f(u(t, x)),\quad t\in \bR,~x\in \bR^5,
\end{equation}
where $f(u) := |u|^{\frac 43}u$. Let $F(u) := \frac 3{10}|u|^{\frac{10}3}$.
The energy functional related to this equation
\begin{equation}
\label{eq:energy}
E(u, \partial_t u) := \int_{\bR^5} \Big(\frac 12 |\partial_t u|^2 + \frac 12 |\grad u|^2 - F(u)\Big)\ud x
\end{equation}
is well-defined for $(u,\partial_t u)\in \dot H^1(\bR^5)\times L^2(\bR^5)$ by the Sobolev inequality
\begin{equation}\label{sobolev}
\|u\|_{L^{\frac{10}3}} \leq C \|\nabla u\|_{L^2}.
\end{equation}
We equip the space of pairs of functions $\vec v=(v, \dot v)$ with the symplectic form
\begin{equation}
\label{eq:symplectique}
\omega(\vec v, \vec w\,) := \la \dot v, w\ra - \la v, \dot w\ra = \la J\vec v, \vec w\,\ra,\quad
J=\begin{pmatrix} 0 & 1\\ -1 & 0\end{pmatrix}.
\end{equation}
Then \eqref{eq:nlw} is the Hamiltonian system corresponding to the Hamiltonian function $E$. In other words
for a solution $u$ of \eqref{eq:nlw}, $\vec u=(u,\partial_t u)$ satisfies
\begin{equation}\label{hamil}
\omega(\vec v, \partial_t \uv \,) = \la \vD E(\uv), \vec v\,\ra,\quad \mbox{for all }\vec v.
\end{equation}
We recall that this equation is locally well-posed in the energy space $\dot H^1(\bR^5)\times L^2(\bR^5)$,
see~\cite{GSV,KM,SS,SSbook} and references therein.
For such solutions, the energy $E(u,\partial_t u)$ is constant in time.

Recall that the function
\[W(x) := \bigg(1 + \frac{|x|^2}{15}\bigg)^{-\frac{3}{2}},\quad x\in \bR^5,
\]
is the ground state solution of the elliptic equation
\begin{equation}\label{elliptic}
\Delta W=W^{\frac 73}\quad \mbox{on $\bR^5$.}
\end{equation}
Up to scaling and translation invariance, $W$ is the unique positive solution of~\eqref{elliptic}.
In particular, $\vec u(t,x)=(W(x),0)$ is a stationary solution of \eqref{eq:nlw} and
other explicit solutions of~\eqref{eq:nlw} are deduced by the sign, scaling, translation and Lorentz invariances of the equation:
\begin{equation*}
\vec u(t,x)=\pm \left(W_{\ell,\lambda}(x-\ell t-x_0),- (\ell\cdot \nabla) W_{\ell,\lambda}(x-\ell t-x_0)\right),
\end{equation*}
where for $\lambda>0$, $x_0\in \bR^5$ and $\ell\in \bR^5$ with $|\ell|<1$,
\begin{equation*}
W_{\lambda,\ell}(x)= W_\lambda \bigg(x+\sigma \frac{\ell(\ell\cdot x)}{|\ell|^2}\bigg) ,
\quad \sigma=\frac 1{\sqrt{1-|\ell|^2}} -1,\quad W_\lambda(x)=\lambda^{-\frac 32} W(\lambda^{-1} x).
\end{equation*}
It is well-known that the ground state $W$ achieves the optimal constant in the critical Sobolev inequality~\eqref{sobolev}, see~\cite{Au,Ta}.
It is also characterized as the threshold element for global existence and scattering (asymptotic linear behavior) of
solutions of~\eqref{eq:nlw}, see~\cite{KM}.
Above this threshold, the study of the large time asymptotic behavior of solutions of \eqref{eq:nlw} raises many questions like the following ones.
\begin{enumerate}[label={(\roman*)}]
\item The classification of all possible long time behaviors of the solutions.
\item The existence and properties of finite or infinite time bubbling solutions.
\item The effect of the nonlinear interactions on the soliton dynamics.
\end{enumerate}
Question (i) is strongly related to the soliton resolution conjecture, 
which predicts that any global bounded solution decomposes asymptotically as $t\to \infty$ into a sum of a finite number $K$ of decoupled energy bubbles plus a solution of the linear wave equation.
Such a decomposition result is proved in~\cite{DKM4} for radially symmetric solutions of the $3$D energy-critical wave equation.
In~\cite{DKM4}, a suitable variant of the decomposition result is also proved for finite time blow-up solutions of
type-II, \emph{i.e.} non ODE type.
In the non radial case, a similar decomposition result (possibly involving excited states, \emph{i.e.} solutions of \eqref{elliptic} other than the ground state) is proved along a subsequence of time for dimensions $3$, $4$,~$5$ in~\cite{DKM6,DJKM1} and extended to any odd dimensions in~\cite{Ro}.
These general results, valid for any initial data, do not specify the number of solitons nor the exact asymptotic behavior of the geometric parameters of each soliton, except a basic decoupling property of the various bubbles and the dispersive part.

Concerning question (ii), several constructions of bubbling solutions with various explicit type-II blow-up rates are available:
see~\cite{DK,KS,KST} in dimension 3, \cite{HR} in dimension~$4$ and~\cite{JJjfa} in dimension~$5$.
In complement to the above mentioned general decomposition results, it is also relevant to study the existence and properties of global solutions whose asymptotic behavior involves several decoupled solitons.
For the energy-critical wave equation in dimension larger than~$6$, a global radial solution decomposing asymptotically as a concentrating bubble on the top of a standing soliton of same sign is constructed in~\cite{JJwave}. Note that this behavior corresponds to a specific choice of sign and blow-up rate; see a nonexistence result in~\cite{JJnon} and a classification result in a similar framework in~\cite{JL}. 
In~\cite{MMwave1}, a solution of \eqref{eq:nlw} containing an arbitrary number~$K$ of bounded traveling solitons is constructed under some restrictions on the speeds $\ell_k$ of the solitons.
We also refer to~\cite{MMwave2} proving inelasticity of soliton interactions in the same context.
Such works clearly relate questions (ii) and (iii)
since the nonlinear interactions between the two solitons are responsible either for the blow up behavior or for the inelasticity property.

We state the main result of this paper.
\begin{theorem}
\label{thm:constr-N5}
Let $K \geq 2$ and $\zz_1, \ldots, \zz_K$ be any $K$ points of $\bR^5$ distinct two by two.
There exist positive constants
 $\cc_1, \ldots, \cc_K$ and a solution $(u,\partial_t u): [0, \infty) \to \dot H^1(\bR^5)\times L^2(\bR^5)$ of \eqref{eq:nlw} such that for all $t>0$,
\begin{equation}
\bigg\|u(t) - \sum_{k=1}^K \frac 1{(c_kt^{-2})^\frac32}W \bigg(\frac{\cdot-z_k}{c_k t^{-2}}\bigg)\bigg\|_{\dot H^1(\bR^5)}+ \|\partial_t u(t)\|_{L^2(\bR^5)}\lesssim t^{-\frac 13}.\end{equation}
\end{theorem}

This result complements the above mentioned articles, 
providing an example of non radial infinite time multiple bubbling in dimension $5$,
in a context where radial multiple bubbling does not seem possible.
Observe that the solutions constructed in Theorem~\ref{thm:constr-N5} only contain bubbles, without any linear remainder, like in~\cite{JJwave,MMwave1}.
Though we do not address uniqueness nor classification questions in this article, we conjecture that $t^{-2}$ is the only possible infinite time blow-up rate for such distant blowing up multiple bubbles.
Theorem~\ref{thm:constr-N5} holds for any set of concentration points $\{z_k\}$,
but the constants $\{c_k\}$ then strongly depend on this choice.
Indeed, in our proof, the determination of suitable constants $\{c_k\}$ is related to the global minimum of some function depending on the distances between the solitons (see Lemma~\ref{le:l}).
Our method of proof should extend to higher space dimensions, however we do not address here the existence of suitable constants $\{c_k\}$ for $N\geq 6$.
We refer to Remark~\ref{rk:heat} for more comments on $\{c_k\}$.

Historically, for nonlinear dispersive equations, the construction of solutions blowing up in finite time at $K$ given points using minimal bubbles was initiated in the case of the mass-critical nonlinear Schr\"odinger equation in~\cite{Mcmp};
see also \cite{MR} for multiple bubble infinite time blow-up.
We refer to~\cite{CM2,MMR2} for recent analogous results for the mass-critical generalized Korteweg-de Vries equation.

Bubbling phenomena were also considered for other energy-critical dispersive or wave models, like 
the wave maps \cite{JJwave,JL,KSTmaps,Pi} and the energy-critical nonlinear Schr\"odinger equation in \cite{JJnls}.
In the parabolic setting, for the energy-critical heat equation in dimension $5$,
 we mention some type-II finite time blow-up results~\cite{dPMW,dPMW3,FHV,Sch}, and infinite time blow-up results~\cite{CDM,dPMW2,Ha}. See Remark~\ref{rk:heat} for a qualitative comparison between results in~\cite{CDM} and Theorem~\ref{thm:constr-N5}.

\subsection{Notation}\label{sec:notation}
In this paper, $\bS^J$ denotes the unit sphere of $\bR^{J+1}$ and $\bar \cB_{\bR^J}$ denotes the unit closed ball of
$\bR^J$. We denote by $\cB(z,r)$ the ball of $\bR^5$ of center $z$ and radius $r\geq 0$.

The bracket $\langle \cdot,\cdot \rangle$ denotes the distributional pairing and the scalar product in $L^2$
and $L^2\times L^2$.

We define a smooth radial cut-off function $\chi$ satisfying $\chi(x) = 0$ for $|x| \geq \frac 23$ and $\chi(x) = 1$ for $|x| \leq \frac 12$ and $0\leq \chi(x)\leq 1$ for $\frac 12 \leq x\leq \frac 23$.

For a function $v:\bR^5\to \bR$ and $\lambda>0$, set
\[
v_{\underline\lambda}(x) := \frac1{\lambda^{\frac 52}} v\left(\frac x\lambda\right),\quad
v_\lambda(x) := \frac1{\lambda^{\frac 32}} v\left(\frac x\lambda\right).
\]
Define
\[\underline \Lambda =\frac 52+x\cdot \nabla,\quad \Lambda =\frac 32+x\cdot \nabla.
\]
For $\vec g=(g,\dot g)$, we denote $\|\vec g\|_\cE=\|\vec g\|_{\dot H^1\times L^2}$. Let
\[
X := (\dot H^1 \cap \dot H^2) \times (L^2 \cap \dot H^1).
\]

\subsection{Finite dimensional dynamics}
\label{ssec:formel1}

Let $ \zz_1,\ldots,\zz_K$ be $K$ points of $\bR^5$ distinct two by two. In this formal discussion, we neglect possible translations of the bubbles and concentrate on the focusing behavior
(this reduction will be justified by the control of translation parameters in the proof of Theorem~\ref{thm:constr-N5}).

For $\bla = (\lambda_1, \ldots, \lambda_K)\in (0,\infty)^K$ and $\bb= (b_1, \ldots, b_K)\in\bR^K$, define
\begin{equation}\label{ansatz}
\WW(\bla, \bb) := \sum_{k} \Big(W_{\lambda_k}(\cdot - \zz_k), b_k \lambda_k^{-1}\Lambda W_{\lambda_k}(\cdot - \zz_k)\Big).
\end{equation}
Here, and in what follows, unless otherwise indicated, sums $\sum_k$ are for indices $k\in\{1,\ldots K\}$.
\begin{remark}
Note that $(W, b\Lambda W)$ is the first-order asymptotic expansion of the self-similar blow-up profile $\WW_b$ for small $b$.
\end{remark}

We take a small number $\epsilon > 0$ and consider the manifold
\begin{equation}
\cM := \{\WW(\bla, \bb): |\bla| + |\bb| < \epsilon\}.
\end{equation}
On this manifold, $(\bla, \bb)$ is a natural system of coordinates. The associated basis of the tangent space is given by
\begin{equation}
\label{eq:basis-TM}
\partial_{\lambda_k} = -\big(\lambda_k^{-1}\Lambda W_{\lambda_k}(\cdot - \zz_k),b_k\lambda_k^{-2}\underline\Lambda\Lambda W_{\lambda_k}(\cdot - \zz_k)\big),\quad
\partial_{b_k} = \big(0, \lambda_k^{-1}\Lambda W_{\lambda_k}(\cdot - \zz_k)\big).
\end{equation}
We wish to compute the restriction of the flow to $\cM$. The Hamiltonian function is
\begin{equation}
E(\bla, \bb) := E(\WW(\bla, \bb)).
\end{equation}
Let
\begin{equation}
\begin{pmatrix} \MMA (\bla, \bb) & \GG(\bla, \bb) \\
-\GG(\bla, \bb) & \MMB(\bla, \bb)\end{pmatrix} =
\begin{pmatrix}(\MMA_{jk})_{j, k=1}^{K} & (\GG_{jk})_{j, k=1}^{K} \\
({-}\GG_{jk})_{j, k=1}^{K} & (\MMB_{jk})_{j, k=1}^{K}\end{pmatrix}\end{equation}
be the matrix of the symplectic form $\omega$ in this basis,
in other words for $j, k \in \{1, \ldots, K\}$,
\begin{align}
& \begin{aligned} \MMA_{jk} = \omega(\partial_{\lambda_j}, \partial_{\lambda_k})& = \lambda_j^{-1}\lambda_k^{-1}\big(b_j\lambda_j^{-1}\la \underline\Lambda\Lambda W_{\lambda_j}(\cdot - \zz_j),
\Lambda W_{\lambda_k}(\cdot - \zz_k)\ra \\
&\quad - b_k\lambda_k^{-1}\la\Lambda W_{\lambda_j}(\cdot - \zz_j), \underline \Lambda\Lambda W_{\lambda_k}(\cdot - \zz_k)\ra\big),\end{aligned}\\
& \GG_{j, k} = \omega(\partial_{\lambda_j}, \partial_{b_k}) = \lambda_j^{-1}\lambda_k^{-1}\la \Lambda W_{\lambda_j}(\cdot - \zz_j), \Lambda W_{\lambda_k}(\cdot - \zz_k)\ra, \\
& \MMB_{j, k} = \omega(\partial_{b_j}, \partial_{b_k})=0.
\end{align}
The motion with constraints is given by the equation
\begin{equation}
\label{eq:constrained}
\begin{pmatrix}\bla' \\ \bb'\end{pmatrix} = \begin{pmatrix}\bs\phi(\bla, \bb) \\ \bs\psi(\bla, \bb)\end{pmatrix} := \begin{pmatrix}\MMA(\bla, \bb) & \GG(\bla, \bb) \\
- \GG(\bla, \bb) & \MMB(\bla, \bb)\end{pmatrix}^{-1}\begin{pmatrix}\partial_{\bla}E(\bla, \bb) \\ \partial_{\bb}E(\bla, \bb)\end{pmatrix}.
\end{equation}
In a suitable regime for $(\bla,\bb)$, we claim
\begin{align}
\partial_{b_k}E(\bla, \bb) &\simeq \|\Lambda W\|_{L^2}^2 b_k \label{eq:dlE} \\
\partial_{\lambda_k}E(\bla, \bb)& \simeq-\|\Lambda W\|_{L^2}^2 B_k(\bla)
\label{eq:dbE}
\end{align}
where
\begin{equation}
B_k(\bla)=- \kappa \lambda_k^\frac{1}{2}\sum_{j\neq k}\Big\{\lambda_j^\frac{3}{2}|\zz_j - \zz_k|^{-3}\Big\}
\quad \mbox{and}\quad \kappa = -\frac73 15^\frac32 \frac {\la\Lambda W, W^\frac 43\ra}{\|\Lambda W\|_{L^2}^2} = \frac{128\sqrt{15}}{7\pi}.
\label{eq:Bk}
\end{equation}
We briefly justify~\eqref{eq:dlE}-\eqref{eq:dbE}.
Using the equation $\Delta W + f(W) = 0$, we have
\begin{equation}
\vD E(\WW(\bla, \bb)) = \Big({-}f\Big(\sum_k W_{\lambda_k}(\cdot - \zz_k)\Big) + \sum_k f(W_{\lambda_k}(\cdot - \zz_k)), \sum_k {b_k}\lambda_k^{-1}\Lambda W_{\lambda_k}(\cdot - \zz_k)\Big).
\end{equation}
We consider cases where $\{\lambda_k\}$, respectively $\{b_k\}$, are asymptotically 
of the size $\lambda(t)>0$, respectively $b(t)$, up to fixed multiplicative constants, where 
$\lambda(t)\to 0$ and $(b/\lambda)(t)\to 0$ as $t\to \infty$.
The first condition means concentration (or ``grow up'') of the solitons while the second condition is natural when searching
polynomial regimes for $\lambda$, since $b$ is related to the time derivative of~$\lambda$.
In such regime, we can easily bound cross terms. 
In particular, from computations similar to that of Lemma~\ref{cl:inter} below, we see that
\begin{equation}
\partial_{b_k}E(\bla, \bb) = \|\Lambda W\|_{L^2}^2 b_k+ O(b\lambda),
\end{equation}
which justifies \eqref{eq:dlE}.

To justify~\eqref{eq:dbE}, we consider again the above expression of $\vD E(\WW(\bla, \bb))$.
The inner product of the first components yields
some constants times $\lambda^2$; the second components yield a constant times $b^2$.
Since we focus on the case $b /\lambda\ll 1$, this second contribution will be negligible with respect to the first.
We thus focus on the first components. We expect the main contribution to come from
\begin{equation}
\sum_{j \neq k}\la f'(W_{\lambda_k}(\cdot - \zz_k))W_{\lambda_j}(\cdot - \zz_j), {\lambda_k^{-1}}\Lambda W_{\lambda_k}(\cdot - \zz_k)\ra.
\end{equation}
Because of the asymptotics $W(x) \simeq 15^{\frac 32}|x|^{-3}$ as $|x| \to \infty$, the factor $W_{\lambda_j}(\cdot - \zz_j)$ can be replaced by the following expression independent of $x$
\begin{equation}
W_{\lambda_j}(\zz_k - \zz_j) = \lambda_j^{-\frac 32}W(\lambda_j^{-1}(\zz_k - \zz_j))
\simeq 15^{\frac 32} \lambda_j^\frac{3}{2}|\zz_k - \zz_j|^{-3}.
\end{equation}
Next, we have
\begin{equation}
\la f'(W_{\lambda_k}(\cdot - \zz_k)), {\lambda_k^{-1}}\Lambda W_{\lambda_k}(\cdot - \zz_k)\ra
= \lambda_k^\frac{1}{2} \la f'(W),\Lambda W\ra=\frac 73 \lambda_k^\frac{1}{2} \la W^{\frac 43},\Lambda W\ra,
\end{equation}
so we obtain \eqref{eq:dbE}.

From~\eqref{eq:dlE}-\eqref{eq:dbE}, we compute the main order terms of $\bs\phi$ and $\bs\psi$.
Again, estimates of cross terms as in the proof of Lemma~\ref{cl:inter}, yield
\begin{equation}
 \GG(\bla, \bb) = \|\Lambda W\|_{L^2}^2\,\tx{Id} + O(\lambda ).
\end{equation}
Thus, using also $\MMB(\bla, \bb)=0$ and the fact that $\MMA(\bla, \bb)$ is of size $b/\lambda$, we obtain
\begin{equation}
\begin{pmatrix}\MMA(\bla, \bb) & \GG(\bla, \bb) \\
- \GG(\bla, \bb) & \MMB(\bla, \bb)\end{pmatrix}^{-1} =
\|\Lambda W\|_{L^2}^{-2} \begin{pmatrix}0 & -\tx{Id} \\ \tx{Id} & 0\end{pmatrix} + O(\lambda)+O(b/\lambda).
\end{equation}
Inserted in \eqref{eq:constrained}, these computations justify the introduction of the following formal system
for the parameters $(\bla, \bb)$:
\begin{equation}
\label{eq:lambda-b-formal}
\left\{\begin{aligned}
\lambda_k'(t) &= -b_k(t) \\
b_k'(t) &= B_k(\bla(t))\,.
\end{aligned}\right.
\end{equation}
By analogy with the differential equation $\lambda''=\lambda^2$, which admits the solution $\lambda(t)= 6 t^{-2}$, we look for a solution of \eqref{eq:lambda-b-formal} of the form
\begin{equation}\label{regime}
 \lambda_k(t)=c_k t^{-2} ,\quad b_k(t)= 2c_k t^{-3},
\end{equation}
for positive constants $c_k$. We need to check that the system \eqref{eq:lambda-b-formal} is actually satisfied for some choice of constants $\{c_k\}$.
The first equation is automatically satisfied by the above expression of $(\lambda_k,b_k)$ and the second one is equivalent to
\begin{equation}
\bs B (\bcc)=-6 \bcc
\end{equation}
where we denote
\begin{equation}
\bcc =(c_1,\ldots,c_K)\quad \mbox{and} \quad\bs B=(B_1,\ldots,B_K).
\end{equation}
We remark that this condition is related to the existence of a critical point
for the following function $V$:
\begin{equation}\label{def:V}
V: \bs \theta=(\theta_1,\ldots,\theta_K)\in \bS_+^{K-1}\mapsto V(\bs \theta)=-\frac 23\kappa \sum_{k}\sum_{j<k} \Big\{ \theta_j^{\frac 32}\theta_k^{\frac 32} |z_j-z_k|^{-3}\Big\},
\end{equation}
where the notation $\bS_+^{K-1}$ means
\[
\bS_+^{K-1}=\Big\{\bs\theta=(\theta_1,\ldots,\theta_K)\in [0,\infty)^K : \sum_{k=1}^K \theta_k^2=1\Big\}.
\]
For later purposes (see Remark~\ref{rk:heat} below), we select a global minimum of the function $V$.
\begin{lemma}\label{le:l}
The following holds
\begin{enumerate}
\item For any $r\geq 0$ and $\bs \theta\in \bS_+^{K-1}$,
\begin{equation}\label{eq:BV}
\bs B(r\bs \theta)=r^2 \nabla V(\bs \theta).
\end{equation}
\item The function $V$ has a global minimum on $\bS_+^{K-1}$, reached at least at a point $\uln{\bs\theta}\in \bS_+^{K-1}$
such that for all $k=1,\ldots,K$, $\uln \theta_k\in (0,1)$. Moreover,
\begin{equation}\label{eq:nV}
\uln n=-\uln {\bs \theta} \cdot \nabla V(\uln {\bs \theta})>0 \quad \mbox{satisfies}\quad
-\nabla V(\uln {\bs \theta})=\uln n \,\uln {\bs \theta}\,.
\end{equation}
\item For $\uln{\bs\theta}\in \bS_+^{K-1}$ and $\uln n>0$ as in {\rm (ii)}, define
\begin{equation}\label{def:c}
\bs c=\uln r \,\uln {\bs \theta} \quad \mbox{where}\quad
\uln r=\frac{6}{-\uln {\bs \theta} \cdot \nabla V(\uln {\bs \theta})}
=\frac {6}{\uln n}.
\end{equation}
Then, it holds $\bs B(\bs c)=-6 \bs c$.
\end{enumerate}
\end{lemma}
\begin{proof}
 (i) follows directly from the definitions of $V$ and $\bs B$.

Proof of (ii). As a nonconstant nonpositive continuous function defined on the compact set $\bS^{K-1}_+$,
the function $V$ has a negative global minimum.
Let $\bs\theta=(\theta_1,\ldots,\theta_K)\in \bS^{K-1}_+$ be such that $\theta_k=0$, for some $k=1,\ldots,K$.
For any $a\in [0,1)$, set
\begin{equation}
\bs \theta(a)=((1-a^{\frac 43})^{\frac 12}\theta_1,\ldots,a^{\frac 23},\ldots,(1-a^{\frac 43})^{\frac 12}\theta_K),\quad
v(a)=V(\bs \theta(a)),
\end{equation}
where the $a^{\frac 23}$ above is located at the $k$th row of the line vector $\bs \theta(a)$.
Observe that $\bs \theta(a)\in \bS^{K-1}_+$ and 
\begin{equation}
v(a)=-\frac 23\kappa\Big[ (1-a^{\frac 43})^{\frac 32}\sum_{\genfrac{}{}{0pt}{1}{j,l\neq k}{j<l}}\Big\{ \theta_j^{\frac 32} \theta_l^{\frac 32}|z_j-z_l|^{-3} \Big\}
+a (1-a^{\frac 43})^{\frac 34} \sum_{j\neq k} \Big\{ \theta_j^{\frac 32}|z_j-z_k|^{-3}\Big\}\Big].
\end{equation}
A simple computation shows that $v'(0)<0$, which proves that the global minimum of the function $V$ on $\bS^{K-1}_+$ is not reached at such $\bs\theta$.

Consider $\uln {\bs \theta}\in \bS_+^{K-1}$ any point of global minimum for $V$.
It follows that there exists
$\uln n\in \bR$ such that $-\nabla V(\uln {\bs \theta})=\uln n\, \uln {\bs \theta}$.
In particular, taking the scalar product by $\uln {\bs \theta}$, we find $-\uln {\bs \theta}\cdot \nabla V(\uln {\bs \theta})=\uln n$, and by~(i) and the expression of $\bs B$, it holds $\uln n>0$.

Proof of (iii). 
Let $\bs c=\uln r \,\uln {\bs \theta}$ where $\uln r$ is defined as in~\eqref{def:c}.
By (i), we have $\bs B(\bcc)=\uln r^2 \nabla V(\uln {\bs \theta})$.
Using also (ii), we obtain $\bs B(\bcc)= -6 \bcc$.
\end{proof}

\begin{remark}\label{rk:heat}
The proof of Theorem~\ref{thm:constr-N5} requires the fact that
$\bcc$ is related to a point of local minimum of $V$ in the interior of $\bS^{K-1}_+$. See Section~\ref{s:3.5}.
The same question in dimension $N\geq 5$ involves the function 
\[
V(\bs \theta)=-C\sum_{k}\sum_{j<k} \Big\{ \theta_j^{\frac {N-2}2}\theta_k^{\frac {N-2}2} |z_j-z_k|^{2-N}\Big\},
\]
where $C>0$. In the proof of (ii) of Lemma~\ref{le:l},
the dimension $N=6$ seems critical in some sense and the fact the global minimum of $V$ is reached only at the interior of $\bS^{K-1}_+$ cannot be proved in the same way for $N\geq 6$. We do not pursue this issue here.

Though some configurations with changing signs seem possible, the proof also uses the fact that the bubbles all have the same sign.
Indeed, only nonlinear interactions of bubbles of same sign have a focusing effect.
See for instance the nonexistence result in \cite{JJnon}.

It is interesting to compare the situation to that of the energy critical nonlinear heat equation considered in \cite{CDM}.
For the latter equation, the bubbling phenomenon involves the same function~$W$. However, soliton-soliton interactions have opposite effects. In~\cite{CDM}, the Dirichlet boundary condition has a focusing effect on the various positive bubbles, and the assumption on the locations of the concentration points ensures that the defocusing effect of the soliton-soliton interactions is lower than the focusing effect of the boundary condition.
This is why the system obtained there (formula (2.19) in \cite{CDM}) is different; in particular, dimension $4$ seems critical and all dimensions higher than $5$ can be treated in a unified way.
\end{remark}

The strategy of the proof of Theorem~\ref{thm:constr-N5} is to construct a solution of \eqref{eq:nlw} converging as 
$t\to \infty$ to the ansatz \eqref{ansatz} with parameters $(\bla,\bb)$ as in \eqref{regime}
and $\bcc$ given by Lemma~\ref{le:l}.

In the next section, we recall coercivity results useful to apply the energy method in a neighborhood of the sum of decoupled solitons. In Section 3, we prove Theorem~\ref{thm:constr-N5}.
\subsection{Acknowledgements}
JJ was partially supported by ANR-18-CE40-0028 project ESSED.
\section{Coercivity results}
\subsection{Single potential}
\label{sec:coercivity}
Linearizing the system \eqref{hamil} around $\vec W=(W,0)$, one obtains
\begin{equation}
\partial_t \gv=J\circ \vD^2 E(\vec W) \vec g=\begin{pmatrix} 0 & \Id\\ -L& 0\end{pmatrix} \vec g,
\end{equation}
where $L$ is the following operator
\[
L g : = -\Delta g - f'(W) g = -\Delta g - \frac 73 W^{\frac43} g.
\]
For $g \in \dot H^1(\bR^5)$ we have the associated quadratic form
\[
\la g, Lg\ra := \int_{\bR^5}\big(|\grad g|^2 - f'(W)g^2\big)\ud x.
\]
\begin{lemma}[{\cite[Appendix D]{Rey}}]\label{le:rey}
If $0\neq g\in \dot H^1(\bR^5)$ satisfies $\la \Delta W, g\ra = \la \Delta \Lambda W, g\ra = \la \Delta \grad W, g\ra = 0$,
then $\la g, Lg\ra > 0$.
\end{lemma}
Since $\la W, LW\ra = -\frac 43\int_{\bR^5}W^{7/3}\ud x < 0$, the operator $L$ has at least one negative eigenvalue.
Denote the smallest eigenvalue $-\nu^2$ $(\nu >0)$ and the corresponding
eigenfunction $\cY$, normalized so that $\|\cY\|_{L^2} = 1$ and $\cY(x) > 0$ for all $x \in \bR^5$.
The facts that $\cY(x) \neq 0$ for all $x \in \bR^5$ and that $\cY$ has exponential decay follow from the general theory of Schr\"odinger operators.

Denote
\begin{equation}
\wt L := L + \nu^2 \la \cY, \cdot \ra \cY,\qquad \la g, \wt L g\ra := \int_{\bR^5}\big(|\grad g|^2 - f'(W)g^2\big)\ud x + \nu^2\la \cY, g\ra^2.
\end{equation}
\begin{lemma}
\label{lem:neg-space}
For all $g \in \dot H^1(\bR^5)$, it holds $\la g, \wt L g\ra \geq 0$. Moreover, $\la g, \wt L g\ra = 0$
if and only if $g \in \spn(\cY, \Lambda W, \partial_{x_1} W, \ldots, \partial_{x_5} W))$.
\end{lemma}
\begin{proof}
Let $g \in \dot H^1(\bR^5)$ and decompose $g = g_1 + g_2$ so that
\begin{gather*}
g_1 \in \spn(\cY, \Lambda W, \partial_{x_1} W, \ldots, \partial_{x_5} W)), \\
\la \Delta W, g_2\ra = \la \Delta\Lambda W, g_2\ra = \la \Delta \grad W, g_2\ra = 0.
\end{gather*}
In order to guarantee that such a decomposition exists, we need to check that the $7\times 7$ matrix
\begin{equation*}
\begin{pmatrix} \la \Delta W, Y\ra & \la \Delta W, \Lambda W\ra & \big(\la \Delta W, \partial_{x_j} W\ra\big)_{j=1,\ldots, 5} \\
\la \Delta \Lambda W, Y\ra & \la \Delta \Lambda W, \Lambda W\ra & \big(\la \Delta \Lambda W, \partial_{x_j} W\ra\big)_{j=1,\ldots, 5} \\
\big(\la \Delta \partial_{x_j} W, Y\ra\big)_{j=1,\ldots, 5} & \big(\la \Delta \partial_{x_j} W, \Lambda W\ra\big)_{j=1, \ldots, 5} & \big(\la \Delta \partial_{x_j} W, \partial_{x_k} W\ra\big)_{j, k=1, \ldots, 5} \end{pmatrix}
\end{equation*}
is non-singular. The upper left term is non-zero because $\Delta W = -f(W) < 0$ and $\cY > 0$.
We also have $\la \Delta W, \Lambda W\ra = 0$ and, using symmetry considerations, we obtain that the matrix
is lower-triangular with non-zero entries on the diagonal.

Since $\wt L g_1 = 0$, using Lemma~\ref{le:rey} we obtain
\[
\la g, \wt L g\ra = \la g_2, \wt L g_2\ra \geq 0,
\]
with equality if and only if $g_2 = 0$.
\end{proof}
\begin{remark}
It follows that $-\nu^2$ is the only negative eigenvalue of $L$.
\end{remark}
\begin{lemma}\label{le:coer0}
There exists $\eta >0$ such that, for any $g\in \dot H^1(\bR^5)$, 
\begin{equation}
\int_{\bR^5} \big(|\nabla g|^2 - f'(W) g^2 \big) \ud x
\geq \eta \|\nabla g\|_{L^2}^2 - \big((\nu^2 + 1)\la \cY, g\ra^2 +\la \Delta\Lambda W, g\ra^2
+ |\la \nabla W, g\ra|^2 \big).
\end{equation}
\end{lemma}
\begin{proof}
If this is false, then there exists a sequence $g_n \in \dot H^1$
such that for $n = 1, 2, \ldots$
\begin{gather*}
\int_{\bR^5}f'(W)g_n^2 \ud x = 1, \\
\label{eq:coer0-1}
\int_{\bR^5} \big( |\nabla g_n|^2 - f'(W) g_n^2 \big) \ud x
\leq \frac 1n \|\nabla g_n\|_{L^2}^2 - \big((\nu^2 + 1)\la \cY, g_n\ra^2 +\la \Delta\Lambda W, g_n\ra^2 
+ |\la \nabla W, g_n\ra|^2 \big) .
\end{gather*}
These inequalities imply in particular that the sequence $(g_n)$ is bounded in $\dot H^1(\bR^5)$.
Upon extracting a subsequence, we can assume $g_n \wto g$ in $\dot H^1(\bR^5)$.
By the Rellich theorem, we have $\int_{\bR^5}f'(W)g^2\ud x = 1$ and thus $g \neq 0$.
Moreover, it holds
\begin{equation}
\lim_{n \to \infty}\la \cY, g_n\ra = \la \cY, g\ra, \quad \lim_{n \to \infty}\la \Delta\Lambda W, g_n\ra = \la \Delta\Lambda W, g\ra, \quad \lim_{n \to \infty}\la \grad W, g_n\ra = \la \grad W, g\ra.
\end{equation}
Hence, by the Fatou property, $g$ satisfies
\begin{gather}
\la g,\wt L g\ra +\la \cY, g\ra^2 + \la \Delta\Lambda W, g\ra^2 + |\la \nabla W, g\ra|^2 \leq 0.
\end{gather}
By Lemma~\ref{lem:neg-space}, this implies
\begin{equation}
g \in \spn(\cY, \Lambda W, \grad W)\quad \mbox{and}\quad 
\la \cY, g\ra = \la \Delta\Lambda W, g\ra =| \la \nabla W, g\ra| = 0.
\end{equation}
This is impossible, since the $7\times 7$ matrix
\begin{equation}
\begin{pmatrix} \la \Delta \Lambda W, \Lambda W\ra & \big(\la \Delta \Lambda W, \partial_{x_j} W\ra\big)_{j=1, \ldots, 5} & \la \Delta \Lambda W, \cY\ra \\
\big(\la \partial_{x_j} W, \Lambda W \ra\big)_{j=1, \ldots, 5} & \big(\la \partial_{x_j} W, \partial_{x_k} W\ra\big)_{j,k=1, \ldots, 5} & \big(\la \partial_{x_j} W, \cY\ra\big)_{j=1, \ldots, 5} \\
\la \cY, \Lambda W\ra & \big(\la \cY, \partial_{x_j} W\ra\big)_{j=1, \ldots, 5} & \la \cY, \cY\ra \end{pmatrix}
\end{equation}
is non-singular (this matrix is upper-triangular with non-zero entries on its diagonal).
\end{proof}
\begin{lemma}\label{le:coer}
For any $\eta > 0$ there exists $R = R(\eta) > 0$ such that for all $g\in \dot H^1(\bR^5)$,
\begin{equation}
\int_{|x|\leq R} |\nabla g|^2 \ud x - \int_{\bR^5} f'(W) g^2 \ud x
\geq -\eta\|\nabla g\|_{L^2}^2 - \nu^2 \la \cY, g\ra^2.
\end{equation}
\end{lemma}
\begin{proof}
By contradiction, suppose there exists $\eta > 0$ and a sequence $g_n \in \dot H^1$ such that it holds $\int_{\bR^5}f'(W)g_n^2 \ud x = 1$ and
\begin{equation}
\int_{|x|\leq n} |\nabla g_n|^2 \ud x - \int_{\bR^5} f'(W) g_n^2 \ud x
\leq -\eta\|\nabla g_n\|_{L^2}^2 - \nu^2 \la \cY, g_n\ra^2.
\end{equation}
In particular, $g_n$ is bounded in $\dot H^1$, and upon extracting a subsequence we can assume that $g_n \wto g \in \dot H^1$.
By Rellich's theorem, $\int_{\bR^5} f'(W)g^2 \ud x = 1$, in particular $g \neq 0$.
We also have $\la \cY, g\ra = \lim_{n \to \infty} \la \cY, g_n\ra$.
Observe that $\textbf{1}_{\{|x| \leq n\}}\grad g_n \wto \grad g$ in $L^2(\bR^5)$, where $\textbf{1}$ denotes the indicator function.
Thus, by the Fatou property, it holds
$ \la g,\wt L g\ra +\eta \|\grad g\|_{L^2}^2 \leq 0$, which contradicts Lemma~\ref{lem:neg-space}.
\end{proof}
\subsection{Multiple potentials}
For $\lambda, \mu \in (0, \infty)$ and $x, y \in \bR^5$ we denote
\begin{equation}
\delta((\lambda, x), (\mu, y)) := \Big|\log\Big(\frac{\lambda}{\mu}\Big)\Big| + \frac{| x - y|}{\lambda}.
\end{equation}
We say that two sequences $(\lambda_n, x_n)$ and $(\mu_n, y_n)$ are \emph{orthogonal} if
\begin{equation}
\lim_{n\to\infty}\delta((\lambda_n, x_n), (\mu_n, y_n)) = \infty.
\end{equation}
Let $K\geq 1$; in what follows $\sum_k$ denotes $\sum_{k=1}^K$.
For $(\lambda^{(k)}, x^{(k)}) \in (0, \infty) \times \bR^5$, we use the notation
\[W^{(k)}(x) := \big(\lambda^{(k)}\big)^{-\frac 32}W\big((x-x^{(k)})/\lambda^{(k)}\big)\]
and similarly for other functions.

\begin{lemma}
\label{le:coer-multi}
There exist $\eta >0$ such that the following holds.
Let $(\lambda^{(k)}, x^{(k)}) \in (0, \infty) \times \bR^5$ for $k = 1, \ldots, K$ satisfy
$\delta((\lambda^{(j)}, x^{(j)}), (\lambda^{(k)}, x^{(k)})) \geq \eta^{-1}$ for all $j \neq k$. 
Let $U \in \dot H^1(\bR^5)$ satisfy
\begin{equation}
\Big\|U - \sum_k W^{(k)}\Big\|_{\dot H^1} \leq \eta.
\end{equation}
Then for any $g\in \dot H^1(\bR^5)$
\begin{multline}
 \int_{\bR^5} \big( |\nabla g|^2 - f'(U) g^2 \big) \ud x
 \geq \eta \|\nabla g\|_{L^2}^2 \\ 
 - \sum_k\Big\{(\nu^2 + 1)\la ( \lambda^{(k)} )^{-2} \cY^{(k)}, g\ra^2 +\la ( \lambda^{(k)} )^{-2}(\Delta \Lambda W)^{(k)}, g\ra^2 + |\la ( \lambda^{(k)} )^{-2} (\nabla W)^{(k)}, g\ra|^2 \Big\}.
\end{multline}
\end{lemma}
\begin{proof}
Assuming that the conclusion fails, we would have sequences $(\lambda_n^{(k)}, x_n^{(k)})$,
$U_n \in \dot H^1(\bR^5)$ and $g_n \in \dot H^1(\bR^5)$ such that
\begin{align*}
&\lim_{n \to \infty} \delta((\lambda_n^{(j)}, x_n^{(j)}), (\lambda_n^{(k)}, x_n^{(k)})) = \infty,\quad\mbox{for } j \neq k, \\
&\lim_{n \to \infty} \Big\|U_n - \sum_k W_n^{(k)}\Big\|_{\dot H^1} = 0,
\end{align*}
and
\begin{multline}
\int_{\bR^5} \big( |\nabla g_n|^2 - f'(U_n) g_n^2 \big) \ud x
\leq \frac 1n \|\nabla g_n\|_{L^2}^2 
\\- \sum_k   \big((\nu^2 + 1)\la ( \lambda^{(k)}  )^{-2}\cY_n^{(k)}, g_n\ra^2 +\la ( \lambda^{(k)}  )^{-2}(\Delta\Lambda W)_n^{(k)}, g_n\ra^2 + |\la ( \lambda^{(k)}  )^{-2}(\nabla W)_n^{(k)}, g_n\ra|^2 \big),
\end{multline}
with the normalization $\int_{\bR^5} f'(U_n)g_n^2\ud x = 1$.
Here, $Y_n^{(k)}=\big(\lambda_n^{(k)}\big)^{-\frac 32}Y\big((x-x_n^{(k)})/\lambda_n^{(k)}\big)$
and similarly for other functions.

The sequence $g_n$ being bounded in $\dot H^1(\bR^5)$, by \cite[Th\'eor\`eme~1.1]{PG98},
upon extracting a subsequence, there exist pairwise orthogonal sequences $(\mu_n^{(j)}, y_n^{(j)})$
and a sequence of \emph{profiles} $\psi^{(j)} \in \dot H^1$ such that
\begin{equation}
\label{eq:pg-decomposition}
g_n = \sum_{j=1}^J \big(\mu_n^{(j)}\big)^{-\frac 32}\psi^{(j)}\big((\cdot - y_n^{(j)})/\mu_n^{(j)}\big) + r_n^{(J)}\quad\mbox{with}\quad \lim_{J \to \infty}\limsup_{n\to\infty}\|r_n^{(J)}\|_{L^\frac{10}{3}} = 0,
\end{equation}
and
\begin{equation}
\label{eq:pyth}
\|g_n\|_{\dot H^1}^2 = \sum_{j=1}^J\|\psi^{(j)}\|_{\dot H^1}^2 + \|r_n^{(J)}\|_{\dot H^1}^2 + o(1)\quad\mbox{as $n \to \infty$}.
\end{equation}
Without loss of generality, we assume that $( y_n^{(j)}, \mu_n^{(j)}) = ( x_n^{(j)}, \lambda_n^{(j)})$ for $j = 1, \ldots, K$.
Indeed, if for some $k \in \{1, \ldots, K\}$ the sequence $( x_n^{(k)}, \lambda_n^{(k)})$ is orthogonal to all the sequences
$( y_n^{(j)}, \mu_n^{(j)})$, we can simply include it in the profile decomposition with identically zero corresponding profile.
If, on the contrary, there exists $j$ such that $( x_n^{(k)}, \lambda_n^{(k)})$ is not orthogonal to $( y_n^{(j)}, \mu_n^{(j)})$,
then, up to extracting a subsequence, we can assume that
\begin{equation}
\lim_{n\to\infty} \lambda_n^{(k)} / \mu_n^{(j)} = \lambda_0 \in (0, \infty)\quad\mbox{and}\quad \lim_{n\to\infty} x_n^{(k)} - y_n^{(j)} = x_0 \in \bR^5.
\end{equation}
Changing $\psi^{(j)}$ if necessary, we can replace $( y_n^{(j)}, \mu_n^{(j)})$ with $( x_n^{(k)}, \lambda_n^{(k)})$.

From $\lim_{n \to \infty} \|U_n - \sum_k W_n^{(k)} \|_{\dot H^1} = 0$ and \eqref{eq:pg-decomposition} we deduce
\begin{equation}
1 = \lim_{n \to \infty} \int_{\bR^5} f'(U_n) g_n^2\ud x = \sum_k \int_{\bR^5}f'(W)\big(\psi^{(k)}\big)^2\ud x.\end{equation}
This shows that at least one of the profiles $\psi^{(1)}, \ldots, \psi^{(K)}$ is not identically zero.
We also have
\begin{align*}
&\lim_{n \to \infty} \big( \lambda^{(k)} \big)^{-2} \la \cY_n^{(k)}, g_n\ra = \la \cY, \psi^{(k)}\ra, \\
& \lim_{n \to \infty} \big( \lambda^{(k)} \big)^{-2}\la (\Delta\Lambda W)_n^{(k)}, g_n\ra = \la \Delta\Lambda W, \psi^{(k)}\ra, \\
 & \lim_{n \to \infty}\big( \lambda^{(k)} \big)^{-2} \la (\grad W)_n^{(k)}, g_n\ra = \la \grad W, \psi^{(k)}\ra.
\end{align*}
The Pythagorean formula \eqref{eq:pyth} thus yields
\begin{align*}
&\sum_k\left\{ \la \psi^{(k)},\wt L \psi^{(k)}\ra
+\big\la \cY, \psi^{(k)}\big\ra^2 + \big\la \Delta\Lambda W, \psi^{(k)}\big\ra^2 
+ \big|\big\la \nabla W, \psi^{(k)}\big\ra\big|^2 \right\}\leq 0.
\end{align*}
This contradicts Lemma~\ref{lem:neg-space}, as in the proof of Lemma~\ref{le:coer0}.
\end{proof}
\section{Construction of multi-bubble solutions}
Let $K\geq 2$ and $\zz_1, \ldots, \zz_K$ be $K$ points of $\bR^5$ distinct two by two.
Set 
\[
d := \frac 12\min_{j\neq k}|\zz_j - \zz_k|>0\quad \mbox{and}\quad \bs \zz=(\zz_1,\ldots,\zz_K).
\]
We consider $\bcc\in (0,\infty)^K$ as given by (iii) of Lemma~\ref{le:l}.
Let $T_0> 1$ to be taken large enough.

\subsection{Modulation and bootstrap}
Let
\[
\bla=(\lambda_1,\ldots,\lambda_K)\in (0,\infty)^K,\quad
\bb=(b_1,\ldots,b_K)\in \bR^K,\quad
\yy=(y_1,\ldots,y_K) \in (\bR^5)^K,
\]
and denote
$\bGa = (\lambda_1,b_1,y_1,\ldots,\lambda_K,b_K,y_K)$.

For all $1\leq k\leq K$, we set
\begin{align*}
W_k &:= W_{\lambda_k }(\cdot-y_k)
=\frac{1}{\lambda_k^{\frac 32}} W\left(\frac{\cdot-y_k}{\lambda_k}\right), \\
\grad_k W_k &:= (\grad W)_{\lambda_k }(\cdot-y_k)
=\frac{1}{\lambda_k^{\frac 32} } \nabla W\left(\frac{\cdot-y_k}{\lambda_k}\right),
\end{align*}
and similarly,
\begin{equation*}
\Lambda_k W_k :=(\Lambda W)_{\lambda_k }(\cdot-y_k ),\quad
\Delta_k \Lambda_k W_k := (\Delta \Lambda W)_{\lambda_k }(\cdot-y_k ), \quad
\YY_k := \YY_{\lambda_k}(\cdot-y_k).
\end{equation*}
Note that the above functions all have the same scaling; in particular,
$\nabla_k W_k=\lambda_k \nabla W_k$.
We also define 
\begin{equation}\label{def:YZk}
\vec Y_k^\pm = (\nu^{-1} Y_k, \pm \lambda_k^{-1}Y_k ),\quad
\vec Z_k^\pm =\frac 12 \lambda_k^{-1} (\nu \lambda_k^{-1} Y_k, \pm Y_k ),\quad
\la \vec Y_k^\pm,\vec Z_k^\pm\ra =1,\quad
\la \vec Y_k^\pm,\vec Z_k^\mp\ra =0.
\end{equation}
Last, we set (recall that $\sum_k$ means $\sum_{k=1}^K$)
\[
\bW=\sum_{k} W_k ,\quad
\bWW =\sum_{k} \vec W_k,\quad \vec W_k=\left(W_k , {b_k }\lambda_k^{-1} \Lambda_k W_k \right).
\]

 The strategy of the proof is to construct solutions $\uv$ of \eqref{eq:nlw} of the form
\begin{equation}\label{eq:dec}
\uv(t )=\bWWt+\gv(t)
\end{equation}
with $\|\gv(t)\|_\cE\ll 1$
on intervals of time $[T_0,T]$, and where the choice of the time-dependent~$\mathcal C^1$ parameter vector $\bGa(t)$ 
 will ensure the orthogonality conditions
\begin{equation}\label{eq:ortho}
\la \Delta_k\Lambda_k W_k , g \ra = 0,\quad
\la \nabla_k W_k , g \ra = 0,\quad
\la \Lambda_k W_k , \dot g \ra = 0
\end{equation}
and will approximately follow the regime~\eqref{regime}.
We denote
\begin{equation}\label{def:a}
a_k^{\pm}:=\la \vec Z_k^\pm,\gv\ra.
\end{equation}

In the next lemma, we construct well-prepared initial conditions at $t=T\geq T_0$ with sufficiently many free parameters $(\alpha_0,\alpha_1,\ldots,\alpha_K)$ related to instabilities (see Remark~\ref{rk:well}).
\begin{lemma}
\label{lem:init-data}
For any $T> T_0$ and any $(\alpha_0,\alpha_1,\ldots,\alpha_K)\in \bar \cB_{\bR^{K+1}}$,
there exists a data $\uv(T)= \uv[T, (\alpha_0,\alpha_1,\ldots,\alpha_K)] \in X$ such that
\begin{equation}
\uv(T) = \WW_{\bGa(T)} + \gv(T),
\end{equation}
with $\bGa(T)$ defined by
\begin{equation}
\label{eq:param-init}
r(T)=|\bcc| T^{-2} + T^{-\frac{12}5} \alpha_0,\quad
\bla(T) = r(T)\frac\bcc{|\bcc|},\quad \bb(T) = 2(r(T))^\frac 32\frac\bcc{|\bcc|^\frac 32}, \quad \bs y(T) = \bzz,
\end{equation}
and $\gv(T)$ satisfies \eqref{eq:ortho} and for all $k=1,\ldots,K$,
\begin{equation}
\label{eq:param-init-bis}
\|\gv(T)\|_\cE \lesssim T^{-4}, \quad \la \vec Z_{k}^+(T),\gv(T)\ra = 0,\quad
\la \vec Z_{k}^-(T),\gv(T)\ra=T^{-4} \alpha_k ,
\end{equation}
where $\vec Z_{k}^\pm(T)$ are defined as in \eqref{def:YZk} for $\Gamma=\Gamma(T)$.

Moreover, $\uv(T)$ is continuous in $X$ with respect to $r_\ini$ and $a_{k,\ini}^-$.
\end{lemma}
 
\begin{proof} 
For $\bGa=\bGa(T)$ fixed as in \eqref{eq:param-init}, we consider $\gv=\gv(T)=(g,\dot g)$ of the form
\begin{equation}
\vec g = \sum_{k} \left\{ b_k^+\vec Y_k^+ +b_k^-\vec Y_k^- + \frac{((c_k\cdot \nabla_k) W _k,0)}{\| \partial_{x_1}W\|_{L^2}^2} + \frac{d_k ( \Lambda_k W_k,0)}{\|\grad \Lambda W\|_{L^2}^2}
+\frac{e_k (0, \Lambda_k W_k)}{\lambda_k \| \Lambda W\|_{L^2}^2}\right\}.
\end{equation}
Consider the linear map $\Psi:(\bR^9)^K\to (\bR^9)^K$ defined as follows:
\begin{align*}
&\Psi\left( (b_k^+,b_k^-,c_k,d_k,e_k)_{k=1,\ldots,K}\right)\\&\qquad=
\left(\la \vec Z_k^+,\vec g\,\ra,\la \vec Z_k^-,\vec g\,\ra,
\lambda_k^{-2} \la \nabla_k W_k, g\ra ,
\lambda_k^{-2} \la \Delta_k \Lambda_k W_k, g\ra ,
\lambda_k^{-1} \la \Lambda_k W_k, \dot g\ra \right)_{k=1,\ldots K}
\end{align*}
It is easy to check that for $T$ large enough the matrix of $\Psi$ is a perturbation of the block matrix
$\diag_K(A)$
where the $9\times 9$ matrix $A$ is upper-triangular with entries $1$ on the diagonal
(the only nonzero entries off the diagonal are due to $\la Y,\Delta \Lambda W\ra \neq 0$).
Moreover,
\[
\Big|\Psi^{-1} \big(\left(0,T^{-4}\alpha_k ,0,\ldots,0 \right)_{k=1,\ldots,K}\big)\Big|\lesssim T^{-4},
\]
and so $\|\vec g\|_\cE\lesssim |(b_k^+,b_k^-,c_k,d_k,e_k)_{k=1,\ldots,K}|\lesssim T^{-4}$.
The continuity property is clear.
\end{proof}

We introduce the following bootstrap estimates
\begin{align}
\|\gv \|_\cE& \leq t^{-\frac{11}3}, \label{eq:g-boot} \\
|\bla - \bcc t^{-2}| &\leq t^{-\frac 73}, \label{eq:lambda-boot} \\
|\bb - 2\bcc t^{-3}| &\leq t^{-\frac{10}{3}}, \label{eq:b-boot} \\
|\bs y - \bzz|& \leq t^{-\frac 73}, \label{eq:y-boot} \\
\sum_{k} (a_k^+)^2 &\leq t^{-8}, \label{eq:ap-boot}
\end{align}
and
\begin{equation}\label{eq:brouwer-boot}
t^{\frac{24}{5}}\big(|\bla | - |\bcc |t^{-2}\big)^2 + t^8 \sum_{k} (a_k^-)^2\leq 1. 
\end{equation}

\begin{remark}\label{rk:well}
The parameters $(\alpha_0,\alpha_1,\ldots,\alpha_K)$ 
and the bootstrap estimate~\eqref{eq:brouwer-boot} are
related to backwards instabilities to be controlled: the backward exponential instability of each soliton (controlled by~$(\alpha_k)_{k=1,\ldots,K}$), and a one-dimensional instability related to the reduced system of ODE, controlled
by $\alpha_0$.
\end{remark}

Let $\uv\in \mathcal C(I_{\max};\dot H^1\times L^2)$ where $I_{\max}\ni T$, be the maximal solution of \eqref{eq:nlw} corresponding to any data $\vec u(T)$ as given by Lemma~\ref{lem:init-data}.
Since $\vec u(T)\in X$, by persistence of regularity (see for instance Appendix~B of \cite{JJjfa}), we have $\uv\in \mathcal C(I_{\max};X)$. Such regularity will allow energy computations without density argument.

Define
\begin{align*} 
T_\star := \inf\{t\in [T_0, T] : 
 &\mbox{ on the interval $[t,T]$, $\vec u$ is well-defined}\\
 &\mbox{ and decomposes as \eqref{eq:dec} where $\Gamma$ and $\vec g$ satisfy
\eqref{eq:g-boot}-\eqref{eq:brouwer-boot}}\}.
 \end{align*}

\begin{lemma}\label{lem:mod}
It holds $T_0\leq T_\star\leq T$ and if $T_\star>T_0$ then
\begin{enumerate}
\item
Equality is reached at $t=T_\star$ in at least one of the inequalities 
\eqref{eq:g-boot}-\eqref{eq:brouwer-boot}.
\item On $[T_\star,T]$, it holds
\begin{align}
|\bla' + \bb | &\lesssim t^{-\frac{11}3}, \label{eq:lambda-mod} \\
|\bs y' | &\lesssim t^{-\frac{11}{3}}, \label{eq:y-mod} \\
|\bb' - \bs B(\bla)| &\lesssim t^{-\frac{14}3}, \label{eq:b-mod}\\
\Big|\left(a_k^\pm\right)' \mp {\nu}\lambda_k^{-1}a_k^\pm\Big| &\lesssim t^{-4},\label{eq:a-mod}
\end{align}
where $\bs B=(B_1,\ldots,B_K)$ is defined by \eqref{eq:Bk}.
\end{enumerate}
\end{lemma}
We begin with a technical lemma.
\begin{lemma}\label{cl:inter}
Under the bootstrap estimates \eqref{eq:g-boot}-\eqref{eq:brouwer-boot}, the following bounds hold, for $j\neq k$,
\begin{equation}\label{eq:inter}\begin{aligned}
& \la\lambda_j^{-1} W_j,\lambda_k^{-1} W_k \ra \lesssim t^{-2},\quad 
 \la W_j^{\frac 53} , W_k^{\frac 53} \ra \lesssim t^{-10}\log t,\\
& \la \lambda_j^{-1} |\nabla_j W_j |,\lambda_k^{-1} |\nabla_k W_k | \ra \lesssim t^{-6},
 \quad \|W_k W_j^{\frac 43}\|_{L^{\frac{10}7}} \lesssim t^{-6}.
\end{aligned}\end{equation}
\end{lemma}
\begin{proof}[Proof of Lemma~\ref{cl:inter}]
First, by change of variable
\begin{equation}
\la\lambda_j^{-1} W_j ,\lambda_k^{-1} W_k \ra= \la\wt\lambda_j^{-1} W_{\wt\lambda_j}(\cdot - t^2 \zz_j), \wt\lambda_k^{-1} W_{\wt\lambda_k}(\cdot - t^2\zz_k)\ra,
\end{equation}
where $\wt\lambda_j := t^2 \lambda_j \sim 1$ and $\wt\lambda_k := t^2 \lambda_k \sim 1$.
The right-hand side term is estimated
by dividing $\bR^5$ into three regions: $D_j := \cB(t^2\zz_j, t^2d)$,
$D_k := \cB(t^2 \zz_k, t^2d)$ and $D_{j,k}=\bR^5\setminus(D_j\cup D_k)$.
In order to estimate the integral outside both balls, we use the bound $|W(x)|\lesssim |x|^{-3}$ and the Cauchy-Schwarz inequality and obtain
\begin{equation}
\int_{D_{j,k}}|y - t^2\zz_j|^{-3}\cdot|y - t^2\zz_k|^{-3}\ud y \lesssim
\int_{t^2d}^\infty r^{-2} \ud r \lesssim t^{-2}.
\end{equation}
For $D_j$, we observe
\begin{equation}
\int_{D_j}|y - t^2\zz_j|^{-3}\ud y \lesssim \int_0^{t^2d}r\ud r \lesssim t^4,
\end{equation}
so using also the trivial $L^\infty$ bound of order $t^{-6}$ for the second factor on $D_j$, we obtain a bound of order $t^{-2}$
for the contribution of $D_j$. This justifies the first bound in \eqref{eq:inter}.

The other estimates in \eqref{eq:inter} are proved similarly, using $|\nabla W(x)|\lesssim |x|^{-4}$.
\end{proof}

In the sequel we will make use of various pointwise estimates obtained from the Taylor expansion of the nonlinearity $f$.
We claim that for all $u, v \in \bR$
\begin{equation}\label{taylor1}
|f(u+v)-f(u)-f(v)-f'(u)v|\lesssim |u|^\frac23|v|^\frac53.
\end{equation}
 To prove \eqref{taylor1}, we consider several cases. If $|v|\leq \frac 12 |u|$, then by Taylor expansion, we have \[
|f(u+v)-f(u)-f(v)-f'(u)v|\lesssim |u|^{\frac 13}|v|^2\lesssim |u|^\frac23|v|^\frac53.
\]
If $\frac 12|u|\leq |v|\leq 2|u|$, then
\[
|f(u+v)|+|f(u)|+|f(v)|+|f'(u)v|\lesssim |u|^{\frac 73}+|v|^{\frac 73}\lesssim |u|^\frac23|v|^\frac53.
\]
Last, if $2|u|\leq |v|$, then 
\[
|f(u+v)-f(v)|+|f(u)|+|f'(u)v|\lesssim |v|^{\frac 43}|u|+|u|^\frac 73+|u|^{\frac 43}|v|\lesssim |u|^\frac23|v|^\frac53.
\]
Next, it is easily checked by induction on $J\geq 1$ that the following holds
\[
\Big| f\Big(\sum_{j=1}^J v_j\Big)-\sum_{j=1}^J f(v_j)\Big|\leq \sum_{j\neq l} |v_j||v_l|^{\frac 43}.
\]
By the triangle inequality and \eqref{taylor1}, we deduce, for any $u, v_j \in \bR$,
\begin{equation}\label{taylor2} 
 \Big|f\Big(u+\sum_{j=1}^Jv_j\Big)-f(u)-\sum_{j=1}^Jf(v_j)-f'(u)\sum_{j=1}^Jv_j\Big|
\lesssim |u|^\frac 23\sum_{j=1}^J|v_j|^\frac 53+\sum_{j\neq l} |v_j||v_l|^{\frac 43}.
\end{equation}
%

\begin{proof}[Proof of Lemma~\ref{lem:mod}] 
At $t=T$, Lemma~\ref{lem:init-data} provides an initial data as in~\eqref{eq:dec}
with the estimates~\eqref{eq:g-boot}-\eqref{eq:brouwer-boot}. Indeed, the assumption~$(\alpha_0,\alpha_1,\ldots,\alpha_K)\in \bar \cB_{\bR^{K+1}}$ implies that~\eqref{eq:brouwer-boot} holds at $t=T$.
This gives \eqref{eq:lambda-boot}-\eqref{eq:b-boot}. Moreover, \eqref{eq:g-boot}, \eqref{eq:y-boot} and~\eqref{eq:ap-boot} are clear from Lemma~\ref{lem:init-data}.

By the local Cauchy theory for~\eqref{eq:nlw}, it is clear that if a solution $\vec u$ satisfies~\eqref{eq:dec}
with~\eqref{eq:g-boot}-\eqref{eq:brouwer-boot} on some interval $[t,T]$, then the solution $\vec u$
also exists on $[t-\tau,T]$, for some $\tau>0$.

To decompose $\vec u(t)$ for $t<T$, the strategy is to
express the orthogonality conditions \eqref{eq:ortho} as a non-autonomous differential system
$\mathbf D \bGa'(t)=\mathbf F(t,\bGa(t))$, where $\mathbf F$ is continuous in $t$ and locally Lipschitz in $\bGa$, and the matrix $\mathbf D$ is a perturbation of the block matrix $\diag_K(D_0) \in \bR^{7K \times 7K}$, where
\[
D_0 = \diag\big( \|\nabla \Lambda W\|_{L^2}^2, \|\Lambda W\|_{L^2}^2, \big(\|\partial_{x_j} W\|_{L^2}^2\big)_{j=1, \ldots, 5}
\big) \in \bR^{7\times 7}.
\]
Then, (i) will follow from the Cauchy-Lipschitz theorem and continuity arguments. 
Moreover, estimates in~(ii) will follow from similar computations combined with~\eqref{eq:g-boot}-\eqref{eq:brouwer-boot}.

Formally, the evolution equation of $\gv(t):=(g,\dot g)(t)$ is
\begin{equation}\label{eq:g}
\partial_t \gv = J\,\vD E (\bWW+\gv)
- \bla' \partial_{\bla}\bWW - \bb'\partial_{\bb}\bWW
-\bs y' \cdot \partial_{\bs y}\bWW 
\end{equation}
which rewrites as
\begin{align}
\label{eq:dtg-1}
\partial_t g &= \dot g + \sum_{k=1}^K \lambda_k^{-1}(\lambda_k' + b_k)\Lambda_k W_k
+ \sum_{k=1}^K \lambda_k^{-1} y_k' \cdot\nabla_k W_k,
\\
 \label{eq:dtg-2}
 \partial_t \dot g &= \Delta g + f(\bW+ g)
- \sum_{k=1}^K f(W_k) \\
&\quad + \sum_{k=1}^K\lambda_k^{-2}\lambda_k'b_k\underline\Lambda_k\Lambda_k W_k
- \sum_{k=1}^K \lambda_k^{-1}b_k'\Lambda_k W_k
+\sum_{k=1}^K \lambda_k^{-2}b_k (y_k'\cdot \nabla_k) \Lambda W_k. 
\end{align}

Proof of \eqref{eq:lambda-mod}-\eqref{eq:y-mod}. We differentiate with respect to time the identity 
$0=\la \lambda_k^{-1}\Delta_k \Lambda_kW_k, g\ra$ which is the first orthogonality condition in \eqref{eq:ortho} 
and we use \eqref{eq:dtg-1}
\begin{equation}
\label{eq:lambda'}
\begin{aligned}
0 &= \dd t \la \lambda_k^{-1}\Delta_k \Lambda_kW_k, g\ra \\
& = \Big\la \lambda_k^{-1} \Delta_k\Lambda_k W_k, \dot g + \sum_j
\lambda_j^{-1}(\lambda_j' + b_j)\Lambda_j W_j
+ \sum_j \lambda_j^{-1}(y_j' \cdot\nabla_j) W_j\Big\ra \\
& \quad- \la \lambda_k' \lambda_k^{-2} \underline\Lambda_k\Delta_k\Lambda_k W_k, g\ra
 -\la \lambda_k^{-2} (y_k'\cdot \nabla_k) \Delta_k\Lambda_k W_k,g\ra.
\end{aligned}
\end{equation}
Rewrite the first term on the right-hand side as
\[
\la \lambda_k^{-1} \Delta_k\Lambda_k W_k, \dot g\ra
=\Big\la \lambda_k^{-1} \Delta_k\Lambda_k W_k, \partial_t u - \sum_k \lambda_k^{-1} {b_k} \Lambda_k W_k\Big\ra.
\]
Note that $\partial_t u$ is continuous in $L^2$ as a function $t$ and 
$\lambda_k^{-1} \Delta_k\Lambda_k W_k$ is locally Lipschitz in $L^2$ as a function of $\bGa$.
Thus, $\la \lambda_k^{-1} \Delta_k\Lambda_k W_k, \partial_t u \ra$ is continuous in $t$ and locally Lipschitz in $\bGa$.
For the second term above, one checks the same properties. Regularity in $t$ and $\bGa$ for all other terms appearing in the computations is proved similarly and omitted.

First, we estimate terms containing $g$ and $\dot g$,
\[
\left|\la \lambda_k^{-1}\Delta_k\Lambda_k W_k, \dot g\ra\right|
\lesssim \| \lambda_k^{-1}\Delta_k\Lambda_k W_k\|_{L^2} \|\dot g\|_{L^2}
\lesssim \|\vec g\|_\cE,
\]
next
\begin{equation}\begin{aligned}
\left|\la \lambda_k' \lambda_k^{-2} \underline \Lambda_k\Delta_k\Lambda_k W_k, g\ra\right|
&\lesssim |\lambda_k'| \|\lambda_k^{-2}\underline \Lambda_k\Delta_k \Lambda_k W_k\|_{L^{\frac{10}{7}}}\|g\|_{L^{\frac{10}{3}}}
\\&\lesssim |\lambda_k'| \|g\|_{\dot H^1}
\lesssim |\bs \lambda'+\bs b| \|\vec g\|_\cE + |\bs b|\|\vec g\|_\cE,
\end{aligned}\end{equation}
and similarly
\[
|\la \lambda_k^{-2} (y_k'\cdot \nabla_k)\Delta_k\Lambda_k W_k,g\ra|\lesssim |\bs y'|\|\vec g\|_\cE.
\]
Next, we claim that matrix $M^\lambda$ with coefficients $m_{jk}^\lambda
:= -\la \lambda_k^{-1} \Delta_k\Lambda_k W_k,\lambda_j^{-1} \Lambda_j W_j\ra$ is diagonally dominant and that its inverse is uniformly bounded. Indeed, for $j=k$, it holds
\[
m_{kk}^\lambda = \la \lambda_k^{-1} \Delta_k\Lambda_k W_k,\lambda_k^{-1} \Lambda_k W_k\ra
=\|\nabla \Lambda W\|_{L^2}^2,
\]
and for $j\neq k$, by \eqref{eq:inter}
\[
|m_{jk}^\lambda|
=\frac{\lambda_j}{\lambda_k}|\la \lambda_k^{-1} \nabla_k\Lambda_k W_k,\lambda_j^{-1} \nabla_j\Lambda_j W_j\ra|
\lesssim t^{-6}.
\]
Last, by symmetry $\la \Delta \Lambda W,\nabla W\ra=0$, and so
\[
\big\la \lambda_k^{-1} \Delta_k \Lambda_k W_k, \lambda_k^{-1}(y_k' \cdot\nabla_k) W_k\big\ra
=0;\]
 for $j\neq k$, by \eqref{eq:inter}
\[
\big| \la \lambda_k^{-1} \Delta_k\Lambda_k W_k, \lambda_j^{-1} (y_j' \cdot\nabla_j) W_j\ra\big|\lesssim t^{-6}|\bs y'|.
\]
Collecting these estimates, using $\|\vec g\|_\cE\lesssim t^{-\frac{11}3}$ and $|\bs b|\ll 1$, we obtain
\begin{equation}\label{sur:la}
|\bs \lambda'+\bs b|\lesssim \left(1+|\bs y'|\right) t^{-\frac {11}3} .
\end{equation}

We differentiate with respect to time the identity 
$0=\la \lambda_k^{-1}\nabla_kW_k, g\ra$ which is the second orthogonality condition in \eqref{eq:ortho} 
and we use \eqref{eq:dtg-1}
\begin{equation}
\label{eq:y'}
\begin{aligned}
0 &= \dd t \la \lambda_k^{-1}\nabla_kW_k, g\ra \\&
=\Big\la\lambda_k^{-1} \nabla_k W_k, \dot g + \sum_j
\lambda_j^{-1}(\lambda_j' + b_j)\Lambda_j W_j+\sum_j\lambda_j^{-1} (y_j' \cdot\nabla_j) W_j\Big\ra \\
&\quad - \la \lambda_k^{-2}\lambda_k' \uln \Lambda_k \nabla_k W_k, g\ra
-\la \lambda_k^{-2}( y_k'\cdot \nabla_k) \nabla_k W_k,g\ra.
\end{aligned}
\end{equation}
First, we have
\[
|\la\lambda_k^{-1} \nabla_k W_k, \dot g\ra|\leq \|\lambda_k^{-1} \nabla_k W_k\|_{L^2} \|\dot g\|_{L^2}
\lesssim \|\vec g\|_\cE.
\]
Second, by $\la \Lambda W,\nabla W\ra=0$ and~\eqref{eq:inter}, we obtain for any $k,j$,
\begin{equation}
 | \la\lambda_k^{-1} \nabla_k W_k, 
\lambda_j^{-1}(\lambda_j' + b_j)\Lambda_j W_j \ra |\lesssim
 t^{-2} |\bs \lambda'+\bs b|.
\end{equation}
Then, for $j=k$, it holds
\[
\la \lambda_k^{-1} \nabla_k W_k,\lambda_k^{-1} (y_k'\cdot \nabla_k) W_k\ra
= y_k' \|\partial_{x_1} W\|_{L^2}^2,
\]
and for $j\neq k$, by \eqref{eq:inter}
\[
|\la \lambda_k^{-1} \nabla_k W_k,\lambda_j^{-1} \nabla_j W_j\ra|
\lesssim t^{-6}.
\]
Next, as before,
\begin{align*}
 |\la \lambda_k^{-2} \lambda_k'\uln\Lambda_k \nabla_k W_k, g\ra |
&\lesssim |\lambda_k'| \|\lambda_k^{-2}\uln \Lambda_k \nabla_k W_k\|_{L^{\frac{10}7}}\|g\|_{L^{\frac{10}3}}
\\
&\lesssim |\lambda_k'| \|g\|_{\dot H^1} 
\lesssim |\bs \lambda'+\bs b| \|\vec g\|_\cE + |\bs b|\|\vec g\|_\cE,
\end{align*}
and
\begin{equation}
 |\la \lambda_k^{-2} (y_k'\cdot \nabla_k) \nabla_k W_k,g\ra|
\lesssim |\bs y'| \|\vec g\|_\cE. 
\end{equation}
Collecting these estimates, using $\|\vec g\|_\cE\lesssim t^{-\frac{11}3}$ and $|\bs b|\ll 1$, we obtain
\begin{equation}\label{sur:y}
|\bs y'|\lesssim t^{-2} |\bs \lambda'+\bs b|+ t^{-\frac{11}3}.
\end{equation}
Combining \eqref{sur:la} and \eqref{sur:y}, we have proved
$
|\bs \lambda'+\bs b|+|\bs y'|\lesssim t^{-\frac{11}3}$, which is \eqref{eq:lambda-mod}-\eqref{eq:y-mod}.

Proof of \eqref{eq:b-mod}. We differentiate with respect to time the identity $\la \Lambda_k W_k ,\dot g \ra=0$ which is the third orthogonality condition in \eqref{eq:ortho} and we use \eqref{eq:dtg-2}
\begin{align*}
0 & = \dd t \la \Lambda_k W_k, \dot g\ra \\&
= {-}\la \lambda_k'\lambda_k^{-1} \Lambda_k^2 W_k, \dot g\ra
-\la \lambda_k^{-1} (y_k'\cdot \nabla_k) \Lambda_k W_k,\dot g\ra \\
&\quad + \Big\la \Lambda_k W_k, \Delta g
+ f(\bW+g)
- \sum_j f\big(W_j\big)\Big\ra \\
&\quad +\Big\la \Lambda_k W_k, \sum_j \lambda_j^{-2} \lambda_j'b_j\underline\Lambda_j\Lambda_j W_j\Big\ra
 -\Big\la \Lambda_k W_k, \sum_j \lambda_j^{-1}b_j'\Lambda_j W_j\Big\ra\\
& \quad +\Big\la \Lambda_k W_k, \sum_j \lambda_j^{-2} b_j (y_j'\cdot \nabla_j) \Lambda W_j\Big\ra.
\end{align*}
For the first two terms, we observe from \eqref{eq:lambda-mod}, $|\bs b|\lesssim t^{-3}$ and $\|\vec g\|_\cE\lesssim t^{-\frac{11}3}$
\[
|\la \lambda_k'\lambda_k^{-1} \Lambda_k^2 W_k, \dot g\ra|
\leq |\lambda_k'| \|\lambda_k^{-1}\Lambda_k^2 W_k\|_{L^2} \|\dot g\|_{L^2}\lesssim 
|\bs \lambda'|\|\vec g\|_\cE\lesssim t^{-\frac {20}3},
\]
and from \eqref{eq:y-mod},
\[
|\la \lambda_k^{-1} (y_k'\cdot \nabla_k) \Lambda_k W_k,\dot g\ra|
\leq \|\lambda_k^{-1} (y_k'\cdot \nabla_k) \Lambda_k W_k\|_{L^2} \|\dot g\|_{L^2}\lesssim 
|\bs y'|\|\vec g\|_\cE\lesssim t^{-\frac {22}3}.
\]
For the next line in the identity above, we set
\[
I_1= \la \Lambda_k W_k, \Delta g+ f(\bW+g)-f(\bW)\ra,\quad
I_2= \Big \la \Lambda_k W_k,f(\bW) - \sum_j f\big(W_j\big)\Big\ra.
\]
We first note that, using the cancellation $L\Lambda W=0$,
\begin{align*}
I_1&=\la \Lambda_k W_k,f(\bW+g)-f(\bW) -f'(W_k) g\ra
\\&=\la \Lambda_k W_k,f(\bW+g)-f(\bW) -f'(\bW) g\ra
+\la \Lambda_k W_k,(f'(\bW)-f'(W_k)) g\ra=I_3+I_4.
\end{align*}
By the Taylor inequality, 
\[
|f(\bW+g)-f(\bW) -f'(\bW) g|\lesssim \bW^{\frac 13} |g|^2 + |g|^{\frac 73},
\]
and so, by Holder and Sobolev inequalities
\begin{equation*}
|I_3|\lesssim \int \bW^{\frac 43} |g|^2 + \int \bW |g|^{\frac 73}
\lesssim \|\bW\|_{\dot H^1}^{\frac 43} \|g\|_{\dot H^1}^{2}
+\|\bW\|_{\dot H^1} \|g\|_{\dot H^1}^{\frac73}
\lesssim t^{-\frac{22}3}.
\end{equation*}
By the Taylor inequality,
\[
\left| f'(\bW) - f'(W_k)\right|\lesssim \sum_{j\neq k} \Big(W_j^{\frac 43} + W_k^\frac13 W_j\Big),
\]
and thus
\begin{equation*}
|I_4|\lesssim \sum_{j\neq k} \int \Big(W_k W_j^{\frac 43}+ W_k^\frac43 W_j\Big) |g|
\lesssim \|g\|_{L^{\frac{10}3}} \sum_{j\neq k}\left( \|W_k W_j^{\frac 43}\|_{L^{\frac{10}7}}+\|W_k^{\frac 43} W_j\|_{L^{\frac{10}7}}\right)
.
\end{equation*}
For $j\neq k$, by~\eqref{eq:inter} we have
 $|I_4|\lesssim t^{-6} \|\vec g\|_\cE\lesssim t^{-\frac{29}3}$.
Therefore, $|I_1|\lesssim t^{-\frac{22}3}$.

We turn to $I_2$ and set
\[
I_2=\Big \la \Lambda_k W_k,f(\bW) - \sum_j f\big(W_j\big) - f'(W_k) \sum_{j\neq k} W_j\Big\ra
+\Big \la \Lambda_k W_k,f'(W_k) \sum_{j\neq k} W_j\Big\ra=I_5+I_6.
\]
Using \eqref{taylor2}, we have
\begin{align*}
|\Lambda W_k| \Big|f(\bW) - \sum_j f\big(W_j\big) - f'(W_k) \sum_{j\neq k} W_j\Big|
&\lesssim W_k^{\frac 53}\sum_{j\neq k} W_j^{\frac 53} + W_k \sum_{j\neq l, j\neq k, l\neq k} W_j W_l^{\frac 43}\\
&\lesssim \sum_{j\neq l} W_j^{\frac 53}W_l^{\frac 53}.
\end{align*}
Thus, using \eqref{eq:inter}, we obtain $|I_5|\lesssim t^{-10} \log t$.

Last, to estimate $I_6$, we only have to consider $\la \Lambda_k W_k , f'(W_k) W_j\ra$ for all $j\neq k$.
By change of variable,
\[
\la \Lambda_k W_k , f'(W_k) W_j\ra=\frac73 \lambda_j^{-\frac 32}\lambda_k^{\frac 32}
\int_{\bR^5} W(x)^\frac 43\Lambda W(x)W\left( \frac{\lambda_k}{\lambda_j} x - \frac{z_j-z_k}{\lambda_j}\right)\ud x.
\]
For $|x|\geq \lambda_k^{-1} d$, it holds by $W(x)\leq |x|^{-3}$ and then Cauchy-Schwarz inequality
\begin{align*}
\int_{|x|\geq \lambda_k^{-1} d} W^\frac 43(x)|\Lambda W(x)| W\left( \frac{\lambda_k}{\lambda_j} x - \frac{z_j-z_k}{\lambda_j}\right)\ud x
&\lesssim \lambda_k^{4} \int W(x)W\left( \frac{\lambda_k}{\lambda_j} x - \frac{z_j-z_k}{\lambda_j}\right)\ud x
\\&\lesssim \lambda_k^{4}\lesssim t^{-8}.
\end{align*}
For $|x|\leq \lambda_k^{-1} d$, it holds
\[
\Big| W\left( \frac{\lambda_k}{\lambda_j} x - \frac{z_j-z_k}{\lambda_j}\right)-W\left(\frac{z_j-z_k}{\lambda_j}\right)\Big|
\lesssim |x| \left|\frac{z_j-z_k}{\lambda_j}\right|^{-4}\lesssim t^{-8} |x|,
\]
and by the explicit expression of $W$, for $|y|\geq 1$,
\[
\left|W(y)-15^\frac 32 |y|^{-3}\right| \leq |y|^{-5}.
\]
We obtain for such $x$,
\[
\Big| W\left( \frac{\lambda_k}{\lambda_j} x - \frac{z_j-z_k}{\lambda_j}\right)-15^\frac32\lambda_j^3 |z_j-z_k|^{-3}\Big|
\lesssim \left|\frac{z_j-z_k}{\lambda_j}\right|^{-5}+t^{-8} |x| \lesssim t^{-8} (1+|x|).
\]
We deduce from these estimates
\begin{align*}
&\Big| \int_{\bR^5} W(x)^\frac 43\Lambda W(x) W\left( \frac{\lambda_k}{\lambda_j}x - \frac{z_j-z_k}{\lambda_j}\right)\ud x
- 15^\frac32\la W^\frac 43,\Lambda W\ra\lambda_j^3 |z_j-z_k|^{-3} \Big|
\\ &\quad \lesssim t^{-8} + t^{-8} \int_{|x|\leq \lambda_k^{-1} d} (1+|x|) W^{\frac 73}(x) \ud x \lesssim t^{-8}.
\end{align*}
Therefore,
\begin{equation*}
\left| \la \Lambda_k W_k , f'(W_k) W_j\ra-\frac73 15^\frac32\la W^\frac 43,\Lambda W\ra\lambda_j^{\frac 32}\lambda_k^{\frac 32} |z_j-z_k|^{-3} \right|
\lesssim t^{-8},
\end{equation*}
and by the definition of $B_k(\bla)$ and $\kappa$ in \eqref{eq:Bk},
\[
|I_6 - \lambda_k\|\Lambda W\|_{L^2}^2 B_k(\bla)|\lesssim t^{-8}.
\]
Next, for $j\neq k$, using \eqref{eq:inter},
\begin{equation*}
|\la \Lambda_k W_k, \lambda_j^{-2} \lambda_j'b_j\underline\Lambda_j\Lambda_j W_j\Big\ra|
\lesssim |\bla'| |\bb| t^{-2}\lesssim t^{-8},
\end{equation*}
while the identity $\la \Lambda W,\underline\Lambda\Lambda W\ra=0$ takes care of the corresponding term for $j=k$.

For the terms $ \la \Lambda_k W_k, \lambda_j^{-1}b_j'\Lambda_j W_j \ra$, we observe if $j=k$ that
\[
\la \Lambda_k W_k, \lambda_k^{-1} \Lambda_k W_k \ra=\lambda_k \|\Lambda W\|_{L^2}^2,
\]
and if $j\neq k$, by \eqref{eq:inter}, $|\la \Lambda_k W_k, \lambda_j^{-1} \Lambda_j W_j \ra|\lesssim t^{-4}$.

Last, for any $j,k$,
\begin{equation*}
 |\la \Lambda_k W_k, \lambda_j^{-2} b_j (y_j'\cdot \nabla_j) \Lambda W_j \ra|
 \lesssim |\bb| |\yy'|\lesssim t^{-\frac{20}3}.
\end{equation*}
Collecting these estimates, we have proved, for all $k=1,\ldots,K$,
\[
|b_k' -B_k(\bla)|\lesssim t^{-\frac {14}3} + t^{-2} |\bb '|,
\]
and since $|B_k(\bla)|\lesssim t^{-4}$, \eqref{eq:b-mod} follows.
 
Proof of~\eqref{eq:a-mod}. By the definition of $a_k^{\pm}$ in~\eqref{def:a}, we compute
$\frac{\ud}{\ud t}a_k^{\pm}= \la \partial_t \vec Z_k^\pm,\gv\ra+ \la \vec Z_k^\pm,\partial_t\gv\ra$.
First,
\[
 \partial_t \vec Z_k^\pm=\lambda_k' \partial_{\lambda_k} \vec Z_k^\pm + y_k' \cdot \partial_{y_k} \vec Z_k^\pm ,
\]
Since $Y$ is exponentially decaying, we obtain from the definition of $\vec Z_k^\pm$, ~\eqref{eq:lambda-mod}-\eqref{eq:y-mod} and~\eqref{eq:g-boot}-\eqref{eq:ap-boot}, the estimate
\[
|\la \partial_t \vec Z_k^\pm,\gv\ra|\lesssim \left(\left|\frac{\lambda_k'}{\lambda_k}\right|
+\left|\frac{y_k'}{\lambda_k}\right|\right) \|\gv \|_\cE
\lesssim t^{-\frac{14}3}. 
\]
Second, using~\eqref{eq:g},
\begin{align*}
\la \vec Z_k^\pm,\partial_t\gv\ra &=
\la \vec Z_k^\pm,J\,\vD E (\bWW+\gv)\ra 
- \la \vec Z_k^\pm,\bla' \partial_{\bla}\bWW \ra
- \la \vec Z_k^\pm,\bb'\partial_{\bb}\bWW\ra
-\la \vec Z_k^\pm,\bs y' \cdot \partial_{\bs y}\bWW \ra
\end{align*}
Using
\begin{equation*}
\bla' \partial_{\bla}\bWW
=\sum_j \begin{pmatrix} \lambda_j^{-1} \lambda_j' \Lambda_j W_j, \lambda_j^{-2} \lambda_j' b_j \underline\Lambda_j \Lambda_j W_j\end{pmatrix},
\end{equation*}
$\la Y_k, \Lambda_k W_k\ra =0$, $|\la \vec Z_k^\pm, \Lambda_j W_j\ra|\lesssim t^{-6}$ for $j\neq k$,
and estimates \eqref{eq:b-boot}, \eqref{eq:lambda-mod}, we obtain
\begin{equation*}
|\la \vec Z_k^\pm,\bla' \partial_{\bla}\bWW \ra|\lesssim t^{-4}.
\end{equation*}
Similarly, using
\begin{equation*}
\bb' \partial_{\bb}\bWW
=\sum_j \begin{pmatrix} 0, \lambda_j^{-1} b_k' \Lambda_j W_j\end{pmatrix},\quad
\bs y' \cdot\partial_{\bs y}\bWW
=\sum_j \begin{pmatrix} \lambda_j^{-1}( y_j'\cdot \nabla_j ) W_j,\lambda_j^{-2} b_j (y_j'\cdot \nabla_j) \Lambda_j W_j\end{pmatrix},
\end{equation*}
$\la Y_k, \nabla_k W_k\ra =0$, and estimates \eqref{eq:b-boot}, \eqref{eq:y-mod}, \eqref{eq:b-mod}, it holds
\begin{equation*}
|\la \vec Z_k^\pm,\bb' \partial_{\bb}\bWW \ra|+|\la \vec Z_k^\pm,\bs y'\cdot \partial_{\bs y}\bWW \ra|\lesssim t^{-4}.
\end{equation*}
Now, we have
\begin{equation*}
J\,\vD E (\bWW+\gv) = \Big (\sum_j \lambda_j^{-1} b_j \Lambda_j W_j , f(\bW+g)-\sum_j f(W_k)
- f'(W_k) g\Big)+ J\, \vD^2 E (W_k)\gv .
\end{equation*}
As before, for all $j$, it holds $|\la \vec Z_k^\pm, ( \lambda_k^{-1} b_j \Lambda_j W_j,0) \ra|\lesssim t^{-4}$,
and arguing as in the proof of \eqref{eq:b-mod}
\begin{equation*}
\Big|\Big\la \lambda_k^{-1} Y_k,f(\bW+g)-\sum_j f(W_k)- f'(W_k) g\Big \ra\Big|\lesssim t^{-4}.
\end{equation*}
Last, we check by direct computations using $L Y = -\nu^2 Y$ that
$\la \vec Z_k^\pm , J\, \vD^2 E (W_k)\gv\ra= \pm {\nu}{\lambda_k^{-1}} a_k^\pm$,
which completes the proof of~\eqref{eq:a-mod}.
\end{proof}
The following statement is the main part of the proof of Theorem~\ref{thm:constr-N5}.
\begin{proposition}
\label{prop:bootstrap}
For any $T > T_0$, there exist $(\alpha_0,\alpha_1,\ldots,\alpha_K)\in \bar \cB_{\bR^{K+1}}$ such that the solution $\uv$
of \eqref{eq:nlw} with data $\uv(T)$ given by Lemma~\ref{lem:init-data} satisfies
$T_\star=T_0$.
\end{proposition}
In Sections~\ref{s:3.4}-\ref{s:3.6}, devoted to the proof of Proposition~\ref{prop:bootstrap}, 
 we tacitly use the following
 direct consequences of~\eqref{eq:g-boot}-\eqref{eq:brouwer-boot} and Lemma~\ref{lem:mod}
\begin{equation}\label{tacit}\begin{aligned}
&\lambda_k(t) \simeq t^{-2},\ b_k(t) \simeq t^{-3},\ |\lambda_k'(t)| \lesssim t^{-3},\ |b_k'(t)| \lesssim t^{-4},\ |y_k'(t)| \lesssim t^{-\frac{11}3},\\
& |\bs B(\bla(t))| \lesssim t^{-4},\ \Big|\frac{\ud}{\ud t}\bs B(\bla(t))\Big| \lesssim t^{-4}.
\end{aligned}
\end{equation}
\subsection{Refined approximate solution}\label{s:3.4}
\begin{lemma}\label{le:8}
There exist smooth radially symmetric functions $Q$, $S$
satisfying on $\bR^5$, for all $\beta\in \bN^5$,
\begin{align*}
LQ &= \frac{105\pi}{128}f'(W) + \Lambda W, \quad |\partial_x^\beta Q(x)| \lesssim |x|^{-1-|\beta|}, \\
LS &= \uln\Lambda \Lambda W, \quad |\partial_x^\beta S(x)| \lesssim |x|^{-1-|\beta|}.
\end{align*}
\end{lemma}
For a proof, see \cite[Proposition 2.1]{JJjfa}.
Note that the explicit constant $\frac{105\pi}{128}$ is related to the orthogonality condition
\[
\Big\langle \frac{105\pi}{128}f'(W) + \Lambda W, \Lambda W\Big\rangle=0.
\]
In the framework of Proposition~\ref{prop:bootstrap}, we set
\[
Q_k:=Q_{\lambda_k}(\cdot - y_k),\quad 
S_k:=S_{\lambda_k}(\cdot - y_k),
\]
\begin{equation}
P:= \sum_k\chi\Big(\frac{\cdot - z_k}{d}\Big)\big(\lambda_k B_k(\bs \lambda )Q_k + b_k ^2 S_k \big),
\end{equation}
where $\chi$ is defined in \S\ref{sec:notation}.
We consider the following refined decomposition of $u$
\begin{equation}
\phi : = \bW + P,\quad h:=g - P
\quad \mbox{so that}\quad u=\bW+g=\phi+h.
\end{equation}
\begin{lemma}
\label{lem:ansatz-err}
Under the bootstrap estimates \eqref{eq:g-boot}-\eqref{eq:brouwer-boot}, it holds
\begin{align}
\|g - h\|_{\dot H^1}& = \|P\|_{\dot H^1} \lesssim t^{-5}, \label{eq:ansatz-err-1} \\
\|\partial_t g - \partial_t h\|_{\dot H^1} &= \|\partial_t P\|_{\dot H^1}\lesssim t^{-6}, \label{eq:ansatz-err-2}
\end{align}
and
\begin{equation}\label{eq:ansatz-err-3} 
\begin{aligned}
\Big\|\partial_t \dot g &- \Big\{\Delta h + f(\phi + h) - f(\phi)
+ \sum_k (\lambda_k' + b_k)b_k\lambda_k^{-2}\uln\Lambda_k\Lambda_k W_k\\
&+\sum_k \lambda_k^{-2}b_k (y_k'\cdot \grad_k) \Lambda_k W_k 
- \sum_k (b_k' - B_k(\bs\lambda))\lambda_k^{-1}\Lambda_k W_k\Big\}\Big\|_{L^2} \lesssim t^{-5}.
\end{aligned}
\end{equation}
\end{lemma}
\begin{proof}
In order to prove \eqref{eq:ansatz-err-1}, note first from \eqref{eq:y-boot} that $|x - y_k| \geq d$ implies $\chi((x-z_k)/d) = 0$,
and thus the Chain Rule yields
\begin{equation}
\|\chi((\cdot - z_k)/d)Q_k\|_{\dot H^1} \lesssim \|Q_{\lambda_k}\|_{L^2(|x| \leq d)}
+ \|\nabla Q_{\lambda_k}\|_{L^2(|x| \leq d)}.
\end{equation}
Using $\lambda_k\lesssim t^{-2}$, we have
\begin{equation}
\label{eq:ansatz-err-1-aux1}
\|Q_{\lambda_k}\|_{L^2(|x| \leq d)} \lesssim \lambda_k \Big(\int_0^{d/\lambda_k}r^{-2}r^4\ud r\Big)^\frac 12 
\lesssim \lambda_k^{-\frac 12} \lesssim t
\end{equation}
and
\begin{equation}
\label{eq:ansatz-err-1-aux2}
\|\nabla Q_{\lambda_k}\|_{L^2(|x| \leq d)} \lesssim \Big(\int_0^{d/\lambda_k}r^{-4}r^4\ud r\Big)^\frac 12
\lesssim \lambda_k^{-\frac 12} \lesssim t.
\end{equation}
Similar estimates involving $S_{\lambda_k}$ hold. Using also $|\lambda_kB_k(\bla)|+|b_k|^2\lesssim t^{-6}$, 
we have proved \eqref{eq:ansatz-err-1}.

In order to bound $\partial_t P$, we write
\begin{equation}
\partial_t (\lambda_k B_k(\bs \lambda) Q_k) = \Big(\dd t(\lambda_k B_k(\bs \lambda))\Big)Q_k
- B_k(\bs \lambda)(y_k'\cdot\grad_k) Q_k 
- B_k(\bs \lambda) \lambda_k' \Lambda_k Q_k.
\end{equation}
Note that the cut-off $\chi\big(\frac{\cdot - z_k}{d}\big)$ is independent of $t$.
For the first term on the right-hand side, the required bound follows from \eqref{eq:ansatz-err-1-aux1}
and $|\dd t(\lambda_k B_k(\bs \lambda))|\lesssim t^{-7}$.
For the second term, we use~\eqref{eq:ansatz-err-1-aux2} (for these terms, we get a stronger bound $\lesssim t^{-6-\frac 83}$).
Finally, the last term is similar to the first one.
Terms involving $S_k$ are bounded similarly.

In view of \eqref{eq:dtg-2}, the refined bound \eqref{eq:ansatz-err-3} is equivalent to
\begin{equation}
\label{eq:ansatz-err-3-0}
 \Big\|\Delta P + f\big(\bW + P\big) - \sum_k f(W_k) -\sum_k \Big(b_k^2\lambda_k^{-2}\uln\Lambda_k \Lambda_k W_k + B_k(\bs \lambda)\lambda_k^{-1}\Lambda_k W_k \Big)\Big\|_{L^2} \lesssim t^{-5}.
\end{equation}
First, consider the complement of the union of the balls $\cB(z_k, d/2)$. In this region all the terms which do not involve $P$ are controlled by $t^{-5}$ in $L^2$ norm (we call such terms negligible).
Indeed, this follows from estimates in \eqref{tacit} and 
\begin{equation}\label{etoile}\begin{aligned}
&\|f(W_k)\|_{L^2(|x-z_k| \geq d/2)} = \lambda_k^{-1}\|f(W)\|_{L^2(|x| \geq d/(2\lambda_k))} \lesssim \lambda_k^{-1}\Big(\int_{d/(2\lambda_k)}^\infty r^{-14}r^4\ud r\Big)^\frac 12 \lesssim \lambda_k^\frac 72, \\
&\|\lambda_k^{-1}\uln \Lambda_k\Lambda_k W_k\|_{L^2(|x-z_k| \geq d/2)} = \|\uln \Lambda\Lambda W\|_{L^2(|x| \geq d/(2\lambda_k))} \lesssim \lambda_k^\frac 12,\\
&\|\lambda_k^{-1}\Lambda_k W_k\|_{L^2(|x-z_k| \geq d/2)} = \|\Lambda W\|_{L^2(|x| \geq d/(2\lambda_k))} \lesssim \lambda_k^\frac 12.
\end{aligned}\end{equation}

Now fix $k \in \{1, \ldots, K\}$ and consider the ball $\cB(z_k, d)$.
We have just seen that in the sum $\sum_{j=1}^K f(W_j)$ only $j = k$ is significant.
Next, we will prove that
\begin{equation}
\label{eq:ansatz-err-3-1}
\Big\|f\big(\bW + P\big) - f(W_k) - f'(W_k) \sum_{j\neq k}W_j - f'(W_k)P\Big\|_{L^2(\cB(z_k, d))} \lesssim t^{-5}.
\end{equation}
Note that in $\cB(z_k, d)$ we have $W_k \gtrsim t^{-3}$, whereas for $j \neq k$ we have $W_j \lesssim t^{-3}$ and $|P| \lesssim t^{-3}$.
From~\eqref{taylor1}, we have
\begin{equation}
|f(u + v) - f(u) - f'(u)v| \lesssim |u|^\frac 13|v|^2 + |v|^\frac{7}{3}.
\end{equation}
Applying this estimate to $u = W_k$ and $v = \sum_{j\neq k}W_j + P$,
so that $|u|^\frac 13|v|^2 + |v|^\frac{7}{3} \lesssim t^{-5}$, and 
integrating over the ball $\cB(z_k, d)$ we get \eqref{eq:ansatz-err-3-1}.

Next, we show that for all $j \neq k$ we have
\begin{equation}
\label{eq:ansatz-err-3-2}
\|f'(W_k)W_j - (15)^\frac 32 \lambda_j^\frac 32 |z_k - z_j|^{-3}f'(W_k)\|_{L^2(\cB(z_k, d))} \lesssim t^{-5}.
\end{equation}
We consider separately $x \in \cB(y_k, \sqrt{\lambda_k})$ and $x \notin \cB(y_k, \sqrt{\lambda_k})$.
In the first case, \eqref{eq:y-boot} yields $|x - z_k| \lesssim t^{-1}$, which implies
\begin{equation}
\Big| \frac{W_j(x)}{W_j(z_k)} - 1 \Big| \lesssim t^{-1}\quad\mbox{and so}\quad W_j(x) = (15)^\frac 32 \lambda_j^\frac 32 |z_k - z_j|^{-3} + O(t^{-4}).
\end{equation}
Since $\|f'(W_k)\|_{L^2} \lesssim \lambda_k^\frac 12 \lesssim t^{-1}$, \eqref{eq:ansatz-err-3-2} is proved for the region $\cB(y_k, \sqrt{\lambda_k})$.

Consider now the region $\cB(z_k, d) \setminus \cB(y_k, \sqrt{\lambda_k})$. We have
\begin{equation}
\|f'(W_k)\|_{L^2(|x -y_k| \geq \sqrt{\lambda_k})} = \sqrt{\lambda_k}\|f'(W)\|_{L^2(|x| \geq \lambda_k^{-1/2})}
\lesssim t^{-1}\Big(\int_{1/\sqrt{\lambda_k}}^\infty r^{-8}r^4\ud r\Big)^\frac 12 \lesssim t^{-1}\lambda_k^\frac 34 \ll t^{-2}.
\end{equation}
Since in $\cB(z_k, d)$ we have $W_j \lesssim t^{-3}$, the proof of \eqref{eq:ansatz-err-3-1} is complete.
Recalling the definition of $B_k(\bs\lambda)$ from~\eqref{eq:Bk}, 
estimate~\eqref{eq:ansatz-err-3-2} can be rewritten as
\begin{equation}
\Big\|f'(W_k)W_j + \frac{105\pi}{128}B_k(\bs\lambda)\lambda_k^{-\frac 12}f'(W_k)\Big\|_{L^2(\cB(z_k, d))} \lesssim t^{-5}.
\end{equation}

Resuming, we have reduced the proof of \eqref{eq:ansatz-err-3} to showing that
\begin{equation}
\label{eq:ansatz-err-3-3}
\begin{aligned}
\Big\| \Delta P + f'(W_k)P -B_k(\bs\lambda)\Big( \frac{105\pi}{128}\lambda_k^{-\frac 12}f'(W_k) +\lambda_k^{-1}\Lambda_k W_k\Big)
 -b_k^2\lambda_k^{-2}\uln\Lambda_k\Lambda_k W_k \Big\|_{L^2(\cB(z_k, d))} \lesssim t^{-5}.
\end{aligned}
\end{equation}
In the region $|x - z_k| \leq \frac d2$, it holds $\chi(\frac{x-z_k}{d})=1$ and $\chi(\frac{x-z_j}{d})=0$
for $j\neq k$; thus the above expression equals $0$ from the definition of $P$
and Lemma~\ref{le:8}.
It remains to show that for the cut-off region $\frac d2 \leq |x - z_k| \leq d$, this term is indeed negligible.
By the estimates \eqref{etoile} dealing with the exterior of the balls $\cB(z_k, d/2)$,
the terms in \eqref{eq:ansatz-err-3-3} not involving $P$ are negligible in this region.
Thus it sufficient to show that
\begin{equation}
\| |\Delta Q_k| + |\grad Q_k| + |Q_k|+ f'(W_k)|Q_k|\|_{L^2(\frac d2 \leq |x - z_k| \leq d)} \lesssim t
\end{equation}
(the terms involving $S_{\lambda_k}$ being bounded analogously). For the four terms above, the inequalities 
\begin{gather}
t^2\Big(\int_{d/(2\lambda_k)}^\infty r^{-6}r^4\ud r\Big)^\frac 12 \lesssim t, \quad
\Big(\int_0^{d/\lambda_k}r^{-4}r^4\ud r\Big)^\frac 12 \lesssim t, \\
t^{-2}\Big(\int_0^{d/\lambda_k}r^{-2}r^4\ud r\Big)^\frac 12 \lesssim t, \quad
t^2\Big(\int_{d/(2\lambda_k)}^\infty (r^{-4}r^{-1})^2 r^4\ud r\Big)^\frac 12 \ll t,
\end{gather}
 provide the desired estimate.
\end{proof}
\subsection{Energy estimates}\label{ssec:3.3}
\begin{lemma}\label{le:a}
Let any $\epsilon>0$ and $R>0$. There exists a radially symmetric function $q=q_{\epsilon,R}\in \cC^{3,1}(\bR^5)$ with the following properties
\begin{enumerate}[label=\emph{(\roman*)}]
\item $q(x)=\frac 12 |x|^2$ for $|x|\leq R$.
\item There exists $\tilde R$ (depending on $\epsilon$ and $R$) such that $q$ is constant for $|x|\geq \tilde R$.
\item $|\nabla q(x)|\lesssim |x|$ and $|\Delta q(x)|\lesssim 1$ for all $x\in \bR^5$, with constants independent of $\epsilon$ and $R$.
\item $\sum_{1\leq j,l\leq 5} (\partial_{x_jx_l} q(x)) v_j v_l \geq -\epsilon \sum_{j=1}^5 |v_j|^2$, for all $x\in \bR^5$, $v\in \bR^5$.
\item $\Delta^2 q(x)\leq \epsilon|x|^{-2}$, for all $x\in \bR^5$.
\end{enumerate}
\end{lemma}
Such a function is constructed in Lemma 4.5 of \cite{JJnls} for dimensions $N\geq 6$, and the construction for $N=5$ follows from arguments in \cite{JJjfa} and \cite{JJnls}.

Fix a function $q$ as in Lemma~\ref{le:a} and define the operators
\begin{align}
[ A_k h](x)&= \frac {3}{10}\frac 1{\lambda_k} \Delta q\Big(\frac{x-y_k}{\lambda_k}\Big) h(x)
+\nabla q \Big(\frac{x-y_k}{\lambda_k}\Big) \cdot \nabla h(x)
,\\
[\underline A_k h](x)&= \frac 12\frac1{\lambda_k} \Delta q\Big(\frac{x-y_k}{\lambda_k}\Big) h(x)
+\nabla q \Big(\frac{x-y_k}{\lambda_k}\Big) \cdot \nabla h(x).
\end{align}
\begin{lemma}\label{le:Ak}
For any $k=1,\ldots,K$, the operators $A_k$ and $\underline A_k$ satisfy the following properties.
\begin{enumerate}[label=\emph{(\roman*)}]
\item The families $\{A_k;\lambda_k>0,y_k\in \bR^5\}$, $\{\underline A_k;\lambda_k>0,y_k\in \bR^5\}$, 
$\{\lambda_k \partial_{\lambda_k} A_k;\lambda_k>0,y_k\in \bR^5\}$, $\{\lambda_k \partial_{\lambda_k}\underline A_k;\lambda_k>0,y_k\in \bR^5\}$, $\{\lambda_k \partial_{y_k} A_k;\lambda_k>0,y_k\in \bR^5\}$ and $\{\lambda_k \partial_{y_k}\underline A_k;\lambda_k>0,y_k\in \bR^5\}$ are bounded in $\mathcal L(\dot H^1,L^2)$,
with norms depending on $q$.
\item For any $g,h\in \dot H^1\cap \dot H^2$,
\[
\la A_k h,f(h+g)-f(h)-f'(h)g\ra = -\la A_kg,f(h+g)-f(h)\ra.
\]
\item For any $\eta>0$, choosing $\epsilon>0$ small enough in Lemma~\ref{le:a}, it holds for all $g\in \dot H^1\cap \dot H^2$,
\[
\la \underline A_k g,\Delta g\ra
\leq \frac{\eta}{\lambda_k} \|g\|_{\dot H^1}^2-\frac 1{\lambda_k} \int_{|x-y_k|<R\lambda_k} |\nabla g(x)|^2\ud x.
\]
\end{enumerate}
\end{lemma}
\begin{proof} (i) Denote
\[
Ah=\frac {3}{10} (\Delta q) h +\nabla q \cdot \nabla h,\quad
 \underline A h = \frac 12 (\Delta q) h +\nabla q \cdot \nabla h.
\]
Since the functions $\Delta q$ and $\nabla q$ have compact supports, it is clear that $A:\dot H^1\to L^2$ is a bounded operator.
For a function $h$, let $h_k(x)= \lambda_k^{-\frac{3}2}h\big(\frac{x-y_k}{\lambda_k}\big)$.
Note that $(A_k h_k)(x)=\lambda_k^{-\frac 52} (Ah)\big(\frac{x-y_k}{\lambda_k}\big)$.
Moreover, $\|A_kh_k\|_{L^2}=\|Ah\|_{L^2}$ and $\|h_k\|_{\dot H^1}=\|h\|_{\dot H^1}$.
Thus, $A_k:\dot H^1\to L^2$ is a bounded operator with the same norm as $A$. The same argument applies to $\underline A_k$ and $\underline A$.

We compute
\begin{align*}
\lambda_k \partial_{\lambda_k}\underline A_k&=-\frac1{2\lambda_k}\Delta q\Big(\frac{x-y_k}{\lambda_k}\Big)
-\frac{1}{2\lambda_k}\frac{x-y_k}{\lambda_k}\cdot \nabla \Delta q\Big(\frac{x-y_k}{\lambda_k}\Big)
-\frac{x-y_k}{\lambda_k}\cdot \nabla^2 q\Big(\frac{x-y_k}{\lambda_k}\Big)\cdot\nabla,\\
\lambda_k\partial_{y_k}\underline A_k&=-\frac1{2\lambda_k}\nabla \Delta q\Big(\frac{x-y_k}{\lambda_k}\Big)
- \nabla^2 q\Big(\frac{x-y_k}{\lambda_k}\Big)\cdot\nabla.
\end{align*}
Thus, the same arguments provide the desired results.

(ii) The relation $\la A h,f(h+g)-f(h)-f'(h)g\ra = -\la A g,f(h+g)-f(h)\ra$ is proved in \cite[Lemma 3.12]{JJwave}, and the relation for $A_k$ follows immediately by change of variable.

(iii) The estimate is proved for $\underline A$ in \cite[Lemma 3.12]{JJwave} and follows for $\underline A_k$ by change of variable.
\end{proof}

We establish energy estimates for the pair $(h, \dot g)$. We define
\begin{equation}
\cI := \int_{\bR^5}\Big\{\frac 12 (\dot g)^2 + \frac 12 |\grad h|^2 - \big(F(\phi + h) - F(\phi) - f(\phi)h\big)\Big\}\ud x,
\end{equation}
and 
\begin{equation}
\cJ_k := -b_k\la \dot g, \uln A_k h\ra.
\end{equation}
Set
\begin{equation}
\cH := \cI + \sum_k\cJ_k.
\end{equation}

\begin{lemma} For any $\delta>0$, choosing $\epsilon>0$ small enough in Lemma~\ref{le:Ak}, it holds
\begin{equation}\label{eq:energyle}
\cH' \geq - \delta t^{-\frac{25}3}.
\end{equation}
\end{lemma}
\begin{proof}
In this proof, the sign ``$\simeq$'' means that equality holds up to error terms of order $t^{-\frac{26}{3}}$.
We call such error terms ``negligible''.

We start by computing $\cI'$. We have
by integration by parts,
\begin{equation}
\label{eq:dtI-0}
\cI' = \la \partial_t \dot g, \dot g\ra - \la \partial_t h, \Delta h + f(\phi + h) - f(\phi)\ra - \la \partial_t \phi, f(\phi+h) - f(\phi) - f'(\phi)h\ra.
\end{equation}

By \eqref{eq:ansatz-err-3} in Lemma~\ref{lem:ansatz-err}, the third orthogonality condition in \eqref{eq:ortho} and \eqref{eq:g-boot}, we have
\begin{equation}
\label{eq:dtgdot}\begin{aligned}
\la \partial_t \dot g, \dot g\ra & \simeq \la \dot g, \Delta h + f(\phi + h) - f(\phi)\ra 
\\& \quad + \sum_k(\lambda_k' + b_k)b_k \lambda_k^{-2}\la \uln\Lambda_k \Lambda_k W_k, \dot g\ra
 +\sum_k b_k\lambda_k^{-2}\la (y_k'\cdot\grad_k) \Lambda_k W_k, \dot g\ra.
\end{aligned}\end{equation}
Moreover, by \eqref{eq:ansatz-err-1}-\eqref{eq:ansatz-err-2} in Lemma~\ref{lem:ansatz-err} and \eqref{eq:g-boot},
\begin{equation}
\label{eq:dtgradh}
\la \partial_t h, \Delta h + f(\phi + h) - f(\phi)\ra \simeq \la \partial_t g, \Delta h + f(\phi + h) - f(\phi)\ra.
\end{equation}
Now, we claim that
\begin{equation}
\label{eq:dtgraph}
\la \partial_t h, \Delta h + f(\phi + h) - f(\phi)\ra \simeq \la \dot g, \Delta h + f(\phi + h) - f(\phi)\ra.
\end{equation}
Note that \eqref{eq:dtg-1} and \eqref{eq:lambda-mod}-\eqref{eq:y-mod} imply
\begin{equation}\label{simple}
\|\partial_t g - \dot g\|_{\dot H^1} \lesssim t^{-\frac 53}.
\end{equation}
Since 
\begin{equation}\label{simple2}
\|f(\phi + h) - f(\phi) - f'(\phi)h\|_{\dot H^{-1}}\lesssim 
\|f(\phi + h) - f(\phi) - f'(\phi)h\|_{L^{\frac{10}7}}\lesssim \|h\|_{\dot H^1}^2 \lesssim t^{-\frac{22}{3}}
\end{equation}
(the last bound follows from \eqref{eq:g-boot} and~\eqref{eq:ansatz-err-1}), we have
\begin{equation}
\begin{aligned}
\la \partial_t g, \Delta h + f(\phi+h) - f(\phi)\ra& = \la \dot g, \Delta h + f(\phi+h) - f(\phi)\ra
+ \la \partial_t g - \dot g, \Delta h + f(\phi+h) - f(\phi)\ra \\
&\simeq \la \dot g, \Delta h + f(\phi+h) - f(\phi)\ra + \la \partial_t g - \dot g, \Delta h + f'(\phi)h\ra.
\end{aligned}
\end{equation}
Now, we check that the last term is negligible.
Fix $k \in \{1, \ldots, K\}$. Using \eqref{eq:dtg-1}, then \eqref{eq:lambda-mod}-\eqref{eq:y-mod}
and the cancellations
$L \Lambda W=0$, $L\nabla W=0$, it is sufficient to prove that
\begin{gather}
\big|\la \Lambda_k W_k, \Delta h + f'(\phi)h\ra\big| = \big|\la \Lambda_k W_k, (f'(\phi) - f'(W_k))h\ra\big| \lesssim t^{-7}, \\
\big|\la \grad_k W_k, \Delta h + f'(\phi)h\ra\big| = \big|\la \grad_k W_k, (f'(\phi) - f'(W_k))h\ra\big| \lesssim t^{-7}.
\end{gather}
Both inequalities will follow from
\begin{equation}
\label{eq:dtgradh-1}
\int_{\bR^5} W_k |f'(\phi) - f'(W_k)||h|\ud x \lesssim t^{-7}.
\end{equation}
In the exterior of all the balls $\cB(z_j, d)$ we have
\begin{equation}
W_k|f'(\phi) - f'(W_k)| \lesssim \sum_j f(W_j)
\end{equation}
and
\begin{equation}
\|f(W_j)\|_{L^{\frac{10}{7}}(|x - z_j| \geq d)} \lesssim \Big(\int_{d/(2\lambda_j)}^\infty r^{-10}r^4\ud r\Big)^\frac{7}{10}
\lesssim \lambda_j^\frac 72 \lesssim t^{-7},
\end{equation}
which yields an estimate better than~\eqref{eq:dtgradh-1} for this region. In the ball $\cB(z_j, d)$ for $j \neq k$ we have
\begin{equation}
|f'(\phi) - f'(W_k)| \lesssim f'(W_j) + f'(P).
\end{equation}
Note that $\|f'(W_j)\|_{L^2} = \sqrt{\lambda_j} \|f'(W)\|_{L^2} \lesssim t^{-1}$. Also, since $\|P\|_{L^\infty} \lesssim t^{-3}$, we obtain
\begin{equation}
\|f'(\phi) - f'(W_k)\|_{L^2(\cB(z_j, d))} \lesssim t^{-1},
\end{equation}
hence H\"older inequality yields
\begin{equation}
\int_{\cB(z_j, d)}W_k |f'(\phi) - f'(W_k)||h|\ud x \lesssim t^{-1}\|W_k\|_{L^5(\cB(z_j, d))}\|h\|_{L^{\frac{10}{3}}(\cB(z_j, d))} \lesssim t^{-4 - \frac{11}{3}} \ll t^{-7},
\end{equation}
which proves \eqref{eq:dtgradh-1} in the ball $\cB(z_j, d)$. In $\cB(z_k, d)$ we write
\begin{equation}
|f'(\phi) - f'(W_k)| \lesssim |f''(W_k)|\Big(|P| + \sum_{j \neq k} W_j \Big) + f'(P) + \sum_{j \neq k}f'(W_j),
\end{equation}
so that in particular
\begin{equation}
W_k|f'(\phi) - f'(W_k)| \lesssim t^{-3}(W_k + f'(W_k)).
\end{equation}
We have $\|f'(W_k)\|_{L^{\frac{10}{7}}} = \lambda_k^{\frac 32} \|f'(W)\|_{L^\frac{10}{7}} \lesssim t^{-3}$
and
\begin{equation}
\|W_k\|_{L^\frac{10}{7}(\cB(z_k, d))} \lesssim \lambda_k^2 \Big(\int_0^{2d/\lambda_k}r^{-\frac{30}{7}}r^4\ud r\Big)^\frac{7}{10}
\lesssim t^{-3},
\end{equation}
hence we obtain by H\"older inequality
\begin{equation}
\int_{\cB(z_k, d)}W_k |f'(\phi) - f'(W_k)||h|\ud x \lesssim t^{-6}t^{-\frac{11}{3}} \ll t^{-7}.
\end{equation}
This finishes the proof of \eqref{eq:dtgradh-1}, which means we have proved \eqref{eq:dtgraph}.

Next, we consider the last term in \eqref{eq:dtI-0}.
Since $\partial_t \phi = \partial_t u - \partial_t h
=\sum_k b_k \lambda_k^{-1}\Lambda_k W_k + \dot g-\partial_t h$, estimates~\eqref{eq:ansatz-err-2} and~\eqref{simple} implies that
\begin{equation}
\Big\|\partial_t \phi - \sum_k b_k \lambda_k^{-1}\Lambda_k W_k\Big\|_{\dot H^1} \lesssim t^{-\frac 53}.
\end{equation}
Thus, using also \eqref{simple2},
\begin{equation}
\la \partial_t \phi, f(\phi + h) - f(\phi) - f'(\phi)h\ra \simeq \sum_k b_k\lambda_k^{-1}\la \Lambda W_k, f(\phi + h) - f(\phi) - f'(\phi)h\ra.
\end{equation}
We conclude that
\begin{equation}\label{surdI}\begin{aligned}
\cI'& \simeq \sum_k(\lambda_k' + b_k)b_k \lambda_k^{-2}\la \uln\Lambda_k \Lambda_k W_k, \dot g\ra
 +\sum_k b_k\lambda_k^{-2}\la (y_k'\cdot\grad_k) \Lambda_k W_k, \dot g\ra 
 \\ & \quad -\sum_k b_k\lambda_k^{-1}\la \Lambda_k W_k, f(\phi + h) - f(\phi) - f'(\phi)h\ra.
\end{aligned}\end{equation}
These remaining terms can only be estimated by $C t^{-\frac {25}3}$, which is the critical size 
for the energy method. Thus, they have to be cancelled by similar terms coming from the virial correction~$\mathcal J$,
see below~\eqref{surdJ}.
The original idea of such a virial correction in a blow-up context is due to \cite{RS} for the mass critical nonlinear Schr\"odinger equation, and was extended to the energy-critical wave and Schr\"odinger equations in~\cite{JJjfa,JJnls}. The presentation here follows closely the one in~\cite{JJjfa,JJnls,JJwave}.

Let $\eta>0$ arbitrarily small. 
We compute $\cJ_k'$ from its definition
\begin{equation}\label{e:dJ}
\cJ_k' = -b_k'\la \dot g, \uln A_k h\ra
-b_k\lambda_k'\la \dot g, (\partial_{\lambda_k}\uln A_k) h\ra
-b_k\la \dot g,y_k'\cdot(\partial_{y_k} \uln A_k) h\ra
-b_k\la \dot g,\uln A_k \partial_th\ra
-b_k\la \partial_t \dot g, \uln A_k h\ra.
\end{equation}
First, by (i) of Lemma~\ref{le:Ak}, \eqref{eq:g-boot}, \eqref{tacit}
and \eqref{eq:ansatz-err-1}, we have
\[
|b_k'\la \dot g, \uln A_k h\ra|+|b_k\lambda_k'\la \dot g, (\partial_{\lambda_k}\uln A_k) h\ra|
+|b_ky_k'\la \dot g,(\partial_{y_k} \uln A_k) h\ra|
\lesssim t^{-4} \|\dot g\|_{L^2}\|h\|_{\dot H^1}
\lesssim t^{-\frac{28}3},
\]
Next, by (i) of Lemma~\ref{le:Ak}, \eqref{eq:g-boot} and \eqref{eq:ansatz-err-2}, we have
\[
|b_k \la \dot g ,\uln A_k (\partial_th-\partial_tg)\ra|
\lesssim t^{-3} \|\dot g\|_{L^2}\|\partial_th-\partial_tg\|_{\dot H^1}
\lesssim t^{-\frac{38}3},
\]
which implies $b_k \la \dot g ,\uln A_k \partial_th \ra\simeq b_k \la \dot g ,\uln A_k \partial_tg\ra$.
Using \eqref{eq:dtg-1} and $\la \dot g,\uln A_k \dot g\ra=0$ (by integration by parts), we have
\begin{equation*}
 \la \dot g ,\uln A_k \partial_tg\ra =
\sum_j \Big\{\lambda_j^{-1}(\lambda_j' + b_j)\la \dot g, \uln A_k \Lambda_j W_j\ra
+ \lambda_j^{-1} \la \dot g, \uln A_k (y_j' \cdot\nabla_j W_j)\ra\Big\}.
\end{equation*}
We first consider $j=k$ in the above sum.
We claim that for $R$ large enough in the choice of $q$ in Lemma~\ref{le:Ak}, it holds
\begin{equation}\label{e:AkLa}
\|\uln A_k \Lambda_k W_k-\lambda_k^{-1}\uln\Lambda_k\Lambda_kW_k\|_{L^2}+
\|\uln A_k \nabla_k W_k-\lambda_k^{-1}\uln\Lambda_k\nabla_kW_k\|_{L^2}\leq \eta.
\end{equation}
Indeed, for $|x|\leq R$, we have $\uln A\Lambda W(x)=\uln\Lambda\Lambda W(x)$, and for $|x|\geq R$, using (iii) of 
Lemma~\ref{le:Ak} and the decay of $W$, we have $|\uln A\Lambda W(x)|+|\uln \Lambda\Lambda W(x)|\leq C |x|^{-3}$.
Thus, $\|\uln A \Lambda W-\uln \Lambda\Lambda W\|_{L^2}\leq C R^{-\frac 12}$, and estimate \eqref{e:AkLa} for
$\Lambda_k W_k$ follows by change of variable. The estimate on $\nabla_k W_k$ is proved similarly.
For $j\neq k$, one checks that $\|\uln A_k \Lambda_j W_j\|_{L^2}+\|\uln A_k \nabla_j W_j\|_{L^2}\lesssim t^{-1}$.

Using also $\uln\Lambda\nabla=\nabla\Lambda$, it follows from what precedes
and $|\bb|\|\dot g\|_{L^2} \lesssim t^{-\frac{25}3}$ that
\begin{equation}\label{import}
\begin{aligned}
\left| b_k \la \dot g ,\uln A_k \partial_th \ra 
-\left\{ b_k\lambda_k^{-2}(\lambda_k' + b_k)\la \dot g, \uln\Lambda_k \Lambda_k W_k\ra
+ b_k\lambda_k^{-2} \la \dot g, (y_k' \cdot\nabla_k \Lambda_k W_k)\ra\right\}\right|
\leq C \eta t^{-\frac{25}3}. 
\end{aligned}
\end{equation}

Finally, we use \eqref{eq:ansatz-err-3}, $\uln A=\frac{\Delta q}{5}+A$ and~(i)-(iii) of Lemma~\ref{le:Ak} to estimate the last term in~\eqref{e:dJ} as follows
\begin{equation}\label{4thterm}\begin{aligned}
&-b_k\la \partial_t \dot g, \uln A_k h\ra
 \geq - \eta b_k\lambda_k^{-1} \|h\|_{\dot H^1}^2\\
&\qquad +b_k \lambda_k^{-1} \bigg\{\int_{|x-y_k|<R\lambda_k} |\nabla h(x)|^2\ud x
	 -\frac 15\Big\la \Delta q\Big(\frac{\cdot-y_k}{\lambda_k}\Big)h,f(\phi+h)-f(\phi) \Big\ra\bigg\}\\
&\qquad +b_k \la A_k \phi , f(\phi+h)-f(\phi)-f'(\phi)h\ra\\
&\qquad -b_k \sum_{j=1}^K (\lambda_j' + b_j)b_j\lambda_j^{-2}\la \uln\Lambda_j\Lambda_j W_j, \uln A_k h\ra
-b_k\sum_{j=1}^K b_j \lambda_j^{-2}y_j' \la \grad_j\Lambda_j W_j, \uln A_k h\ra \\
&\qquad +b_k \sum_{j=1}^K (b_j' - B_j(\bs\lambda))\lambda_j^{-1}\la \Lambda_j W_j, \uln A_k h\ra
-C t^{-\frac {35}3}.
\end{aligned}\end{equation}
The first line of \eqref{4thterm} is lower bounded by $-C \eta t^{-\frac{25}3}$.
For the second line, we first observe that since $|\Delta q(x)|\lesssim 1$ for all $x\in \bR^5$,
using also \eqref{simple2}, we have
\begin{equation}\label{stra}
\Big| \Big\la \Delta q \Big(\frac{\cdot-y_k}{\lambda_k}\Big)h,f(\phi+h)-f(\phi) \Big\ra
- \Big\la \Delta q \Big(\frac{\cdot-y_k}{\lambda_k}\Big)h,f'(\phi)h \Big\ra\Big|
\lesssim t^{-11}.
\end{equation}
We claim that
\begin{equation}\label{simple3}
\int_{|x-y_k|<\lambda_k \tilde R} |f'(\phi)-f'(W_k)|h^2\lesssim t^{-\frac{31}3},\quad 
\int_{|x-y_k|>\lambda_k R} |f'(W_k)| h^2 \lesssim R^{-2}t^{-\frac{22}3}.
\end{equation}
Indeed, by Holder and Sobolev inequalities, and then Taylor expansion
\begin{align*}
\int_{|x-y_k|<\lambda_k \tilde R} |f'(\phi)-f'(W_k)|h^2
& \lesssim \|h\|_{\dot H^1}^2\|f'(\phi)-f'(W_k)\|_{L^{\frac {10}3}(|x-y_k|<\lambda_k \tilde R)}\\
& \lesssim \|h\|_{\dot H^1}^2 \sum_{j\neq k} \left(\|W_j\|_{L^{\frac {10}3}(|x-y_k|<\lambda_k \tilde R)}^{\frac 43}
+\|W_jW_k^{\frac 13}\|_{L^{\frac 52}(|x-y_k|<\lambda_k \tilde R)}
\right)\\
& \lesssim \|h\|_{\dot H^1}^2 \sum_{j\neq k} \|W_j\|_{L^{\frac {10}3}(|x-y_k|<\lambda_k \tilde R)} 
\lesssim \|h\|_{\dot H^1}^2|\bla|^{\frac 32} \lesssim t^{-\frac{31}3}.
\end{align*}
Similar estimates give $\int_{|x-y_k|>\lambda_k R} |f'(W_k)| h^2 \lesssim \|h\|_{\dot H^1}^2R^{-2} \lesssim R^{-2}t^{-\frac{22}3}$ and thus \eqref{simple3} is proved.

From the definition of $q$ in Lemma~\ref{le:Ak}, $\Delta q(x)=5$ for $|x|\leq R$ and $\Delta q(x)=0$ for $|x|\geq \tilde R$.
Thus, \eqref{stra} and~\eqref{simple3} imply that
\[
\Big|\frac 15\Big\la \Delta q \Big(\frac{\cdot-y_k}{\lambda_k}\Big)h,f(\phi+h)-f(\phi) \Big\ra
-\int_{|x-y_k|<\lambda_k R} f'(W_k(x))h^2(x)\ud x \Big| \lesssim R^{-2} t^{-\frac{25}3}.
\]
Therefore, up to negligible terms, the second line of \eqref{4thterm} is estimated by 
\[
b_k \lambda_k^{-1}\int_{|x-y_k|<R\lambda_k} \left\{|\nabla h(x)|^2 - f'(W_k(x))h^2(x)\right\} \ud x
-C R^{-2} t^{-\frac{25}3}.
\]
Using~\eqref{eq:ansatz-err-1}, \eqref{eq:ap-boot}-\eqref{eq:brouwer-boot} and the
definitions of $Z_k^\pm$, it holds
\begin{equation}\label{h:ortho:Y}
|\la\lambda_k^{-2} Y_k,h\ra |^2 \lesssim \|h-g\|_{\dot H^1}^2+|\la \lambda_k^{-2} Y_k,g\ra|^2
\lesssim t^{-10}+ (a_k^-)^2 +(a_k^+)^2 \lesssim t^{-8}.
\end{equation}
Thus,
 applying Lemma~\ref{le:coer} to $h$,
with $R$ large enough, we have the lower bound
\begin{equation*}
 \int_{|x-y_k|<R\lambda_k} \left\{|\nabla h(x)|^2 - f'(W_k(x))h^2(x)\right\} \ud x
\geq - \eta \|\nabla h\|_{L^2}^2 - Ct^{-8}
\geq - 2 \eta t^{-\frac {22}3} .
\end{equation*}
Next, we claim that 
\begin{equation}\label{simple4}
\|\lambda_k \uln A_k\phi - \uln \Lambda_k W_k\|_{L^{\frac{10}3}} \lesssim t^{-1} + R^{-2},
\end{equation}
which, combined with \eqref{simple2}, implies that the third term in the right-hand side of~\eqref{4thterm} is equal to 
\[
b_k\lambda_k^{-1} \la \Lambda_k W_k , f(\phi+h)-f(\phi)-f'(\phi)h\ra
\]
up to negligible terms.
To prove \eqref{simple4}, we just observe that since $\uln A_k W_k = \lambda_k^{-1} \Lambda_k W_k$ for $|x-y_k|\leq R\lambda_k$
and $\uln A_k\phi =\uln A_kW_k =0$ for $|x-y_k|\geq \tilde R\lambda_k$, it holds
\begin{align*}
\|\lambda_k \uln A_k\phi - \Lambda_k W_k\|_{L^{\frac{10}3}}
&\leq \|\lambda_k \uln A_k(\phi - W_k)\|_{L^{\frac{10}3}( |x-y_k|< \tilde R \lambda_k)}+
\|\lambda_k \uln A_k W_k - \Lambda_k W_k\|_{L^{\frac{10}3}(|x-y_k|>R\lambda_k)}\\
&\lesssim t^{-1} + R^{-2}.
\end{align*}

Finally, we claim that the last three terms of~\eqref{4thterm} are negligible.
Indeed, this is a consequence of Lemma~\ref{lem:mod}, (i) of Lemma~\ref{le:Ak} and the bound $\|h\|_{\dot H^1}
\lesssim t^{-\frac{11}3}$.

We conclude that
\begin{equation}\label{surdJ}\begin{aligned}
\cJ_k' &\leq 
- (\lambda_k' + b_k)b_k \lambda_k^{-2}\la \uln\Lambda_k \Lambda_k W_k, \dot g\ra
- b_k\lambda_k^{-2}\la (y_k'\cdot\grad_k) \Lambda_k W_k, \dot g\ra 
 \\ & \quad + b_k\lambda_k^{-1}\la \Lambda_k W_k, f(\phi + h) - f(\phi) - f'(\phi)h\ra
- C(\eta+R^{-2}) t^{-\frac {25}3} .
\end{aligned}\end{equation}
Combining~\eqref{surdI} and \eqref{surdJ}, we obtain, with $\delta > 0$ arbitrarily small and under the bootstrap assumptions, that
$
\cH' \geq -\delta t^{-\frac{25}{3}},
$ which is \eqref{eq:energyle}.
\end{proof}
\subsection{Control of the scaling parameters}\label{s:3.5} 
In this subsection, we prove that for all $t \in [T_\star,T]$,
\begin{align}
|\bs \lambda - \bs c t^{-2}| &\leq \frac 12t^{-\frac 73} , \label{eq:lambda-boot-2} \\
|\bs b - 2 \bs c t^{-3}| &\leq \frac 12t^{-\frac{10}{3}} .\label{eq:b-boot-2}
\end{align}
The argument is one of the original aspects of this article compared to previous works on multi-solitons.
Equations~\eqref{eq:lambda-mod} and \eqref{eq:b-mod} are necessary but not sufficient to estimate~$\bs\lambda$
and $\bs b$. 
Indeed, to control a one-dimensional instability related to $|\bs \lambda|$,
we need to use specific approximate Lyapunov functionals $\mathcal F$ and $\mathcal G$ and 
the following bound on $|\bla |$ from~\eqref{eq:brouwer-boot}
\begin{equation}
\label{eq:r-bound}
\left| |\bla|-|\bs c| t^{-2}\right| \leq t^{-\frac{12}{5}},\quad \mbox{for all }t \in [T_\star, T].
\end{equation}
Recall that~\eqref{eq:brouwer-boot} gathers all terms for which a topological argument is required
(see next subsection).

\begin{proof}[Proof of \eqref{eq:lambda-boot-2}-\eqref{eq:b-boot-2}]
For $t \in [T_\star, T]$ denote $r := |\bs \lambda|$, and define $\bs \theta\in {\mathbb S}^{K-1}_+$, $\rho\in \bR$ and $\bs b^\perp\in \bR^K$
by the relations
\begin{equation}
\bs\lambda = r\bs\theta,\quad \bs b = \rho\bs\theta + \bs b^\perp, \quad \bs b^\perp \perp \bs\theta.
\end{equation}
Note that from~\eqref{eq:param-init}, $\bs b^\perp(T) = 0$ and $\bs\theta(T) = \bs c/|\bs c|$.
We will prove that for all $t \in [T_\star, T]$
\begin{gather}
\label{eq:b-perp-boot}
|\bs b^\perp| \leq t^{-\frac{31}{9}}, \\
\label{eq:theta-boot}
|\bs \theta - \bs c/|\bs c|| \leq t^{-\frac{4}{9}}.
\end{gather}
Projecting \eqref{eq:lambda-mod} first on $\bs \theta$, and then on its orthogonal complement, we obtain, for all $t \in [T_\star, T]$,
\begin{gather}
\label{eq:dtr}
|r' + \rho| \lesssim t^{-\frac{11}{3}}, \\
\label{eq:dttheta}
|\bs \theta' + \bs b^\perp/r| \lesssim t^{-\frac{5}{3}}.
\end{gather}
From \eqref{eq:b-mod} and (i) of Lemma~\ref{le:l} we get
\begin{equation}
\label{eq:dtbperp}
(\bs b^\perp)' = \bs b' - (\rho\bs \theta)' = r^2 \grad V(\bs \theta) - \rho'\bs \theta - \rho\bs\theta' + O(t^{-\frac{14}{3}}).
\end{equation}

Consider the following quantity
\begin{equation}
\cF := \frac 12 r^{-3}|\bs b^\perp|^2 + V(\bs \theta).
\end{equation}
Let
\begin{equation}
\tau := \inf\{t\in [T_\star,T]: \mbox{\eqref{eq:b-perp-boot} and \eqref{eq:theta-boot} hold on $[t, T]$}\}
\end{equation}
and suppose that $\tau > T_\star$.
We check that for all $t \in [\tau, T]$ we have
\begin{equation}
\label{eq:dtI-intern}
\cF' \gtrsim -t^{-\frac{19}{9}}.
\end{equation}
Indeed, we have
\begin{equation}
\label{eq:dtI-intern-2}
\cF' = -\frac 32 r'r^{-4}|\bs b^\perp|^2 + r^{-3}(\bs b^\perp)'\cdot \bs b^\perp + \bs\theta'\cdot \grad V(\bs\theta).
\end{equation}
From \eqref{eq:b-perp-boot}, \eqref{eq:dtr} and \eqref{eq:r-bound} we obtain
\begin{equation}
\label{eq:dtI-intern-3}
\Big| r' r^{-4}|\bs b^\perp|^2 + r^{-4}\rho |\bs b^\perp|^2\Big| \lesssim t^{-\frac{11}{3} + 8 - \frac{62}{9}}
= t^{-\frac{23}{9}} \ll t^{-\frac{19}{9}}.
\end{equation}
From \eqref{eq:b-perp-boot}, \eqref{eq:r-bound} and \eqref{eq:dtbperp} we obtain
\begin{equation}
\label{eq:dtI-intern-4}
\big|r^{-3}(\bs b^\perp)'\cdot \bs b^\perp - r^{-3}\big(r^2 \grad V(\bs \theta) - \rho'\bs \theta - \rho\bs\theta'\big)\cdot \bs b^\perp\big| \lesssim t^{6 - \frac{14}{3} - \frac{31}{9}} = t^{-\frac{19}{9}}.
\end{equation}
Using $\bs\theta\cdot \bs b^\perp = 0$ and
\begin{equation}
\big|r^{-3}\rho\bs\theta'\cdot \bs b^\perp+ r^{-4}\rho|\bs b^\perp|^2\big| \leq r^{-3}\rho|\bs b^\perp|\big|\bs\theta'+\bs b^\perp/r\big| \lesssim t^{6-3-\frac{31}{9}-\frac 53} = t^{-\frac{19}{9}},
\end{equation}
this yields
\begin{equation}
\big|r^{-3}(\bs b^\perp)'\cdot \bs b^\perp - \big(r^{-1} \grad V(\bs \theta)\cdot \bs b^\perp + r^{-4}\rho|\bs b^\perp|^2\big)\big| \lesssim t^{-\frac{19}{9}}.
\end{equation}
Since $\bs c/|\bs c|^{-1}$ is a critical point of $V\vert_{\bS^{K-1}_+}$ and $V$ is smooth in its neighborhood
(see Lemma~\ref{le:l}),
\eqref{eq:theta-boot} implies that the component of $\nabla V(\bs\theta)$ orthogonal to $\bs\theta$ is $O(t^{-\frac 49})$.
Thus \eqref{eq:dttheta} yields
\begin{equation}
\label{eq:dtI-intern-5}
|\bs\theta'\cdot \grad V(\bs\theta) - ({-}\bs b^\perp/r)\cdot\grad V(\bs\theta)| = |(\bs\theta' + \bs b^\perp/r)\cdot \grad V(\bs \theta)| \lesssim t^{-\frac 49 - \frac 53} = t^{-\frac{19}{9}}.
\end{equation}
Formula \eqref{eq:dtI-intern-2} and the bounds \eqref{eq:dtI-intern-3}, \eqref{eq:dtI-intern-4} and \eqref{eq:dtI-intern-5} yield
\begin{equation}
\label{eq:dtI-intern-6}
\cF' = \frac 32 r^{-4}\rho|\bs b^\perp|^2 + r^{-1}\grad V(\bs\theta)\cdot \bs b^\perp + r^{-4}\rho|\bs b^\perp|^2
-r^{-1}\grad V(\bs\theta)\cdot \bs b^\perp + O(t^{-\frac{19}{9}}),
\end{equation}
which proves \eqref{eq:dtI-intern}, because $\rho > 0$.

Integrating \eqref{eq:dtI-intern} between $t \in [\tau, T]$ and $T$ yields
\begin{equation}
\frac 12 r^{-3}|\bs b^\perp|^2 + V(\bs\theta) - V(\bs c/|\bs c|) \lesssim t^{-\frac{10}{9}}.
\end{equation}
Since $V$ attains its global minimum at $\bs c / |\bs c|$, this implies
\begin{equation}
\label{eq:bperp-improved}
r^{-3}|\bs b^\perp|^2 \lesssim t^{-\frac{10}{9}}\quad \mbox{and so}\quad |\bs b^\perp| \lesssim t^{-3-\frac 59} = t^{-\frac{32}{9}}.
\end{equation}
Thus \eqref{eq:dttheta} implies $|\bs \theta'| \lesssim t^{-\frac{14}{9}}$; in particular, using
$\theta(T)=\frac{\bcc}{|\bcc|}$, we obtain the following improved bound on $[\tau, T]$
\begin{equation}
\label{eq:theta-improved}
|\bs \theta - \bs c/|\bs c|| \lesssim t^{-\frac 59}.
\end{equation}
Bounds \eqref{eq:bperp-improved} and \eqref{eq:theta-improved} show that \eqref{eq:b-perp-boot} and \eqref{eq:theta-boot}
cannot break down at $t = \tau$, thus proving that \eqref{eq:b-perp-boot} and \eqref{eq:theta-boot} indeed hold on $[T_\star, T]$.

By the triangle inequality, \eqref{eq:r-bound} and \eqref{eq:theta-boot} we have
\begin{equation}
|\bs\lambda - \bs c t^{-2}| = |r\bs\theta - \bs c t^{-2}| \leq |r - |\bs c|t^{-2}| + |\bs c|t^{-2}|\bs \theta - \bs c/|\bs c|| \lesssim t^{-\frac{12}{5}} + t^{-2}t^{-\frac 49} \ll t^{-\frac 73},
\end{equation}
which proves \eqref{eq:lambda-boot-2}.

Now, we analyse the evolution of $(r, \rho)$. For $\bs \theta \in \bS^{K-1}$ set
\begin{equation}
n(\bs\theta) := -\bs\theta\cdot \grad V(\bs\theta).
\end{equation}
Taking the inner product of \eqref{eq:dtbperp} with $\bs\theta$ gives
\begin{equation}
\rho' = -r^2 n(\bs\theta) - (\bs b^\perp)'\cdot \bs\theta + O(t^{-\frac{14}{3}}).
\end{equation}
Note that \eqref{eq:dttheta} and \eqref{eq:b-perp-boot} imply in particular $|\bs \theta'| \lesssim t^{-\frac{31}{9} + 2} = t^{-\frac{13}{9}}$.
Since $\bs b^\perp\cdot \bs\theta = 0$ for all $t$, we have
\begin{equation}
\big|(\bs b^\perp)'\cdot \bs\theta\big| = \big|\bs b^\perp\cdot \bs\theta'\big|
\leq |\bs b^\perp|\big|\bs\theta'\big|\lesssim t^{-\frac{31}{9} - \frac{13}{9}} \ll t^{-\frac{14}{3}},
\end{equation}
and thus
\begin{equation}
\rho' = -r^2 n(\bs \theta) + O(t^{-\frac{14}{3}}).
\end{equation}
Since $n(\bs\theta)$ is smooth in a neighborhood of $\uln{\bs\theta} = \bs c/|\bs c|$ (see Lemma~\ref{le:l}),
\eqref{eq:theta-boot} yields the estimate $|n(\bs \theta) - n(\bs c/|\bs c|)| \lesssim t^{-\frac 49}$. Thus
\begin{equation}
\label{eq:dtrho}
\rho' = -r^2 n(\bs c/|\bs c|) + O(t^{-\frac{40}{9}}).
\end{equation}

Consider 
\[\cG := \frac 12 \rho^2 - 2|\bs c|^{-1}r^3.\]
Using $r \lesssim t^{-2}$, $\rho \lesssim t^{-3}$,
\eqref{eq:dtr}, \eqref{eq:dtrho} and the fact that $|\bs c| = 6(n(\bs c/|\bs c|))^{-1}$ 
(see \eqref{def:c}), we compute
\begin{equation}
\cG' = -r^2\rho n(\bs c/|\bs c|) + 6|\bs c|^{-1}r^2\rho + O(t^{-\frac{40}{9}}t^{-3} + t^{-4}t^{-\frac{11}{3}}) = O(t^{-\frac{67}{9}}).
\end{equation}
Since $\cG(T) = 0$ by \eqref{eq:param-init}, we obtain by integration on $[t,T]$,
\begin{equation}
\big(\rho - 2|\bs c|^{-\frac 12}r^\frac 32\big)\big(\rho + 2|\bs c|^{-\frac 12}r^\frac 32\big) = 2\cG
\lesssim t^{-\frac{58}{9}};
\end{equation}
thus \eqref{eq:r-bound} yields
\begin{equation}
\label{eq:rho-bound}
\big|\rho - 2|\bs c|^{-\frac 12}r^\frac 32\big| \lesssim t^{-\frac{31}{9}},
\end{equation}
and last \eqref{eq:dtr} implies
\begin{equation}
\label{eq:dtr-1st-ord}
\big|r' + 2|\bs c|^{-\frac 12}r^\frac 32\big| \lesssim t^{-\frac{31}{9}}.
\end{equation}
The bound \eqref{eq:rho-bound} also implies, again using \eqref{eq:r-bound},
\begin{equation}
\label{eq:rho-bound-2}
\big|\rho - 2|\bs c|t^{-3}\big| \lesssim t^{-\frac{31}{9}} + t^{-3}\big|(t^2 r)^{\frac 32} - |\bs c|^\frac 32\big|
\lesssim t^{-\frac{31}{9}} + t^{-3 - \frac 25} \lesssim t^{-\frac{17}{5}}.
\end{equation}
By the triangle inequality and previous estimates, we have
\begin{equation}
\begin{aligned}
|\bs b - 2\bs c t^{-3}| &= |\rho\bs\theta + \bs b^\perp - 2\bs c t^{-3}| \leq |\bs b^\perp| + \big|\rho - 2|\bs c|t^{-3}\big| + 2|\bs c|t^{-3}\big|\bs \theta - \bs c/|\bs c|\big| \\
&\lesssim t^{-\frac{31}{9}} + t^{-\frac{17}{5}} + t^{-\frac{31}{9}} \ll t^{-\frac{10}{3}},
\end{aligned}
\end{equation}
which proves \eqref{eq:b-boot-2}.\end{proof}

\subsection{Closing the bootstrap argument}\label{s:3.6} 
Now, we prove that for all $t \in [T_\star,T]$, it holds
\begin{align}
\|\vec g\|_\cE &\leq \frac 12t^{-\frac{11}3}, \label{eq:g-boot-2} \\
|\bs y - \bs z| &\leq \frac 12t^{-\frac 73} , \label{eq:y-boot-2} \\
\sum_{k=1}^K (a_k^+)^2 &\leq \frac 12t^{-8} . \label{eq:ap-boot-2}
\end{align}
\begin{proof}[Proof of \eqref{eq:g-boot-2}-\eqref{eq:ap-boot-2}]
Using~\eqref{eq:param-init-bis} and~\eqref{eq:ansatz-err-1}, $\|h(T)\|_{\dot H^1} \leq \|g(T)\|_{\dot H^1} + \|h(T) - g(T)\|_{\dot H^1} \lesssim T^{-4}$, and thus it holds
\begin{equation}
\cH(T) \lesssim T^{-8}.
\end{equation}
Hence, integration of~\eqref{eq:energyle} implies $\cH(t) \leq \delta t^{-\frac{22}{3}}$ 
and thus $\cI(t)\leq 2 \delta t^{-\frac{22}3}$, for all $t \in [T_\star, T]$.

From \eqref{eq:ortho} and~\eqref{eq:ansatz-err-1}, it holds 
\[|\la \lambda_k^{-2}\Delta_k\Lambda_k W_k, h\ra| + |\la \lambda_k^{-2}\grad_k W_k, h\ra|\lesssim t^{-5}.\]
Besides, from~\eqref{h:ortho:Y}, we recall that 
$|\la \lambda_k^{-2}\YY_k, h\ra|^2 \lesssim t^{-8}$.
Therefore, applying Lemma~\ref{le:coer-multi} and standard arguments
to estimate $\int | F(\phi+h)-F(\phi)-f(\phi)h-f'(\phi)h^2|\ll t^{-\frac{22}3}$, we obtain the following estimate, for $\delta$ small enough, 
\[
\|\grad h\|_{L^2}^2+\|\dot g\|_{L^2}^2 \lesssim \cI + t^{-8}\leq \frac 13 t^{-\frac{22}{3}},
\]
This yields \eqref{eq:g-boot-2} on $\vec g$ using again the estimate on $g-h$ from \eqref{eq:ansatz-err-1}.

Bound \eqref{eq:y-boot-2} follows immediately from \eqref{eq:y-mod}, $\bs y(T)=\bzz$ (see \eqref{eq:param-init}) and integration.
In order to prove \eqref{eq:ap-boot-2}, we observe that \eqref{eq:a-mod} and \eqref{eq:ap-boot} yields
\begin{equation}
\Big|\dd t\sum_k (a_k^+)^2 - 2\nu\sum_k\frac{(a_k^+)^2}{\lambda_k}\Big| \lesssim t^{-8},
\end{equation}
hence, by~\eqref{eq:lambda-boot} there is $C > 0$ (independent of $t$) such that
\begin{equation}
\label{eq:dtap-pos}
\dd t \sum_k (a_k^+)^2 \geq C t^2 \sum_k (a_k^+)^2 + O(t^{-8}).
\end{equation}
It is clear that \eqref{eq:ap-boot-2} holds for $t$ close to $T$.
Supposing that \eqref{eq:ap-boot-2} breaks down for the first time at some $T_1\in (T_\star,T)$,
we would have on the one hand $\dd t \sum_{k=1}^K |a_k^+(T_1)|^2 \leq 0$; on the other hand
\eqref{eq:dtap-pos} would yield $\dd t\sum_{k=1}^K |a_k^+(T_1)|^2 > 0$. This contradiction proves \eqref{eq:ap-boot-2}.\end{proof}

Finally, we complete the proof of Proposition~\ref{prop:bootstrap}, dealing with the remaining
bootstrap estimate~\eqref{eq:brouwer-boot}.
For the sake of contradiction, suppose that for any $(\alpha_0, \alpha_1, \ldots, \alpha_K) \in \bar \cB_{\bR^{K+1}}$, 
it holds $T_\star=T_\star(\alpha_0, \alpha_1, \ldots, \alpha_K) \in (T_0,T]$.
It follows from \eqref{eq:lambda-boot-2}-\eqref{eq:b-boot-2} and \eqref{eq:g-boot-2}-\eqref{eq:ap-boot-2} that 
on $[T_\star,T]$, equality is reached in none of the estimates
\eqref{eq:g-boot}-\eqref{eq:ap-boot}.
Therefore, from (i) of Lemma~\ref{lem:mod}, equality has to be reached at $t=T_\star$ in estimate~\eqref{eq:brouwer-boot}.

Recall that $r:=|\bla|$ and set also
\[\wt a_0(t) := t^\frac{12}{5}(r(t) - |\bs c|t^{-2}),\quad\wt a_k(t) := t^4 a_k^-(t)
\]
so that from \eqref{eq:param-init}-\eqref{eq:param-init-bis},
\[
\wt a_0(T) = \alpha_0\quad\mbox{and}\quad \wt a_k(T) = \alpha_k \mbox{ for $k=1,\ldots K$.}
\]
The contradiction assumption says that for any $(\alpha_0, \alpha_1, \ldots, \alpha_K) \in \bar \cB_{\bR^{K+1}}$, it holds
\begin{equation}
\mbox{for all $t \in[T_\star, T]$, $(\wt a_0(t), \wt a_1(t), \ldots, \wt a_K(t)) \in \bar \cB_{\bR^{K+1}}$
and 
$(\wt a_0(T_\star), \wt a_1(T_\star), \ldots, \wt a_K(T_\star)) \in\bS^K$.}
\end{equation}
Consider the application $\Phi: \bar \cB_{\bR^{K+1}}\to \bS^K$ defined by
\begin{equation}
\Phi(\alpha_0, \alpha_1, \ldots, \alpha_K) := (\wt a_0(T_\star), \wt a_1(T_\star), \ldots, \wt a_K(T_\star)).
\end{equation}
To prove that $\Phi$ is continuous, we only need to check that $(\alpha_0, \alpha_1, \ldots, \alpha_K)\mapsto T_\star$ is continuous.
This property is deduced from the following transversality condition:
 for any $T_1\in [T_\star,T]$ such that $(\wt a_0(T_1), \wt a_1(T_1), \ldots, \wt a_K(T_1)) \in\bS^K$, it holds
\begin{equation}
\label{eq:transve}
\sum_{k=0}^K \wt a_k'(T_1) \wt a_k(T_1) < 0.
\end{equation}
Proof of~\eqref{eq:transve}.
On the one hand, for $k = 1, \ldots, K$, estimate \eqref{eq:a-mod} yields
\begin{equation}
(a_k^-)'(T_1)a_k^-(T_1) = -\frac{\nu}{\lambda_k(T_1)}\big(a_k^-(T_1)\big)^2 + O(T_1^{-4}|a_k^-(T_1)|),
\end{equation}
and so
\begin{equation}
\label{eq:ak'ak}
\wt a_k'(T_1)\wt a_k(T_1) = T_1^8(4 T_1^{-1} a_k^-(T_1) + (a_k^-)'(T_1))a_k^-(T_1) = -\frac{\nu}{2\lambda_k(T_1)}\wt a_k(T_1)^2 + O(|\wt a_k(T_1)|).
\end{equation}
Using $|O(|\wt a_k(T_1)|)|\leq \frac{\nu}{4\lambda_k(T_1)}\wt a_k(T_1)^2 + CT_1^{-2}$ for some constant $C$ and taking the sum for $k=1,\ldots,K$, we obtain
\begin{equation}
\label{eq:ak'ak-sum}
\sum_{k=1}^K \wt a_k'(T_1) \wt a_k(T_1)
\leq -\sum_{k=1}^K \frac{\nu}{4\lambda_k(T_1)}\wt a_k(T_1)^2 + C T_1^{-2}.
\end{equation}
On the other hand, using the definition of $\wt a_0$ and then \eqref{eq:dtr-1st-ord}, it holds
\begin{equation}
\begin{aligned}
\wt a_0'(T_1) &= \frac{12}{5}T_1^\frac 75(r(T_1) - |\bs c|T_1^{-2}) + T_1^\frac{12}{5}(r'(T_1) + 2|\bs c|T_1^{-3}) \\
&= \frac{12}{5}T_1^\frac 75(r(T_1) - |\bs c|T_1^{-2}) - 2|\bs c|^{-\frac 12}T_1^\frac{12}{5}(r(T_1)^\frac 32 - |\bs c|^\frac 32 T_1^{-3}) + O(T_1^{-\frac{47}{45}}).
\end{aligned}
\end{equation}
Observe that \eqref{eq:r-bound} implies
\begin{equation}
2|\bs c|^{-\frac 12}T_1\frac{r(T_1)^\frac 32 - |\bs c|^\frac 32 T_1^{-3}}{r(T_1) - |\bs c|T_1^{-2}} = 2|\bs c|^{-\frac 12}\frac{r(T_1)T_1^2 + (r(T_1)T_1^2)^\frac 12 |\bs c|^\frac 12 + |\bs c|^\frac 32}{(r(T_1)T_1^2)^\frac 12 + |\bs c|} = 3 + O(T_1^{-\frac 25}),
\end{equation}
so that
\begin{equation}
\wt a_0'(T_1) = -\frac 35 T_1^{-1}a_0(T_1) + O(T_1^{-\frac{47}{45}}).
\end{equation}
This estimate combined with \eqref{eq:ak'ak-sum} and $\sum_{k=0}^K\wt a_k(T_1)^2=1$ yield
\begin{equation}
\sum_{k=0}^K \wt a_k'(T_1)\wt a_k(T_1) \leq -\frac 35 T_1^{-1}\sum_{k=0}^K\wt a_k(T_1)^2 + O(T_1^{-\frac{47}{45}}) = -\frac 35 T_1^{-1} + O(T_1^{-\frac{47}{45}}) < 0,
\end{equation}
provided that $T_1$ is large enough, which proves \eqref{eq:transve}.

Therefore, $\Phi$ is continuous on $\bar \cB_{\bR^{K+1}}$ and its restriction to $\bS^K$ is the identity.
This is a contradiction with the no-retraction theorem.

\subsection{Proof of Theorem~\ref{thm:constr-N5} from Proposition~\ref{prop:bootstrap}}
We follow the strategy by compactness from~\cite{CMM,JJjfa,JJwave,Ma1,Mcmp,RS}, using the uniform estimates of Proposition~\ref{prop:bootstrap} on a sequence of well-prepared solutions of~\eqref{eq:nlw}.

Consider the solution $\vec u_n$ given by Proposition~\ref{prop:bootstrap} 
for $T=T_n$ where $T_n:=n>T_0$.
On the interval $[T_0,T_n]$, this solution is well-defined and its decomposition 
$(\Gamma_n,\vec g_n)$ satisfies the uniform estimates~
\eqref{eq:g-boot}-\eqref{eq:y-boot}.
In particular, from $\vec u_n = \vec W_{\Gamma_n}+\vec g_n$, we check that,
for all $t\in[T_0,T_n]$
\begin{equation}\label{unif:final}
\bigg\|u_n(t) - \sum_k \frac 1{(c_kt^{-2})^\frac32}W \bigg(\frac{\cdot-z_k}{c_k t^{-2}}\bigg)\bigg\|_{\dot H^1 }
+\|\partial_t u_n(t)\|_{L^2 } \leq C t^{-\frac 13}.
\end{equation}
We take a possibly larger $T_0$ so that $CT_0^{-\frac 13}<\eta$ where $\eta>0$ is the constant of 
Proposition~\ref{pr:A2}.

Since the sequence $(\vec u_n(T_0))_n$ is bounded in $\dot H^1\times L^2$,
after extraction of a subsequence, there exists $\vec u_0$ in $\dot H^1\times L^2$ such that
$\vec u_n(T_0)\rightharpoonup \vec u_0$ weakly in~$\dot H^1\times L^2$.
Fix $T>T_0$.
From Proposition~\ref{pr:A2} applied to the compact set
\[
\mathcal K=\Big\{\Big(\sum_k \frac 1{(c_kt^{-2})^\frac32}W \bigg(\frac{\cdot-z_k}{c_k t^{-2}}\bigg) ,0\Big),~
 t\in [T_0,T] \Big\},
\]
the solution $\vec u$ of \eqref{eq:nlw} corresponding to $\vec u(T_0)=u_0$ is well-defined and it holds
$\vec u_n(t)\rightharpoonup \vec u(t)$ weakly in~$\dot H^1\times L^2$ on $[T_0,T]$. By~\eqref{unif:final} and the properties of weak convergence, the solution
$\vec u$ satisfies, for all $t\in [T_0,T]$
\[
\bigg\|u(t) - \sum_k \frac 1{(c_kt^{-2})^\frac32}W \bigg(\frac{\cdot-z_k}{c_k t^{-2}}\bigg)\bigg\|_{\dot H^1 }
+\|\partial_t u(t)\|_{L^2 } \lesssim t^{-\frac 13}.
\]
Since $T\geq T_0$ is arbitrary, the solution $\vec u$ is defined and 
satisfies the conclusion of Theorem~\ref{thm:constr-N5} on~$[T_0,\infty)$.
We obtain a solution defined on $[0,\infty)$ with similar properties by time translation. 

\appendix

\section{Weak continuity of the flow near a compact set}

We reproduce two statements from Appendix A.2 of \cite{JJwave}
with the only difference that they are given here for general solutions and not only for radially symmetric solutions.
Using the result of profile decomposition stated in \cite[Proposition 2.8]{DKM1}, the proofs are similar up to dealing with additional position parameter.

\begin{proposition}\label{pr:A1}
There exists a constant $\eta>0$ such that the following holds. Let $\vec u:[t_0,T_{\max})\to \dot H^1\times L^2$
be a maximal solution of \eqref{eq:nlw} with $T_{\max}<\infty$.
Then for any compact set $\cK\subset \dot H^1\times L^2$ there exists $\tau<T_{\max}$ such that
$\dist(\vec u(t),\cK)>\eta$ for all $t\in [\tau,T_{\max})$.
\end{proposition}

\begin{proposition}\label{pr:A2}
There exists a constant $\eta>0$ such that the following holds. Let $\cK\subset \dot H^1\times L^2$ be a compact set and let 
$\vec u:[T_1,T_2]\to \dot H^1\times L^2$ be a sequence of solutions of \eqref{eq:nlw} such that
\[\dist(\vec u_n(t),\cK)\leq\eta,\quad \mbox{for all $n\in \bN$ and $t\in [T_1,T_2]$.}\]
Suppose that $u_n(T_1)\rightharpoonup \vec u_0$ weakly in $\dot H^1\times L^2$. Then the solution $\vec u(t)$ of \eqref{eq:nlw}
with the initial condition $\vec u(T_1)=\vec u_0$ is defined for $t\in [T_1,T_2]$ and
\[
\vec u_n(T_1)\rightharpoonup \vec u(t),\quad \mbox{weakly in $\dot H^1\times L^2$ for all $t\in [T_1,T_2]$.}
\]
\end{proposition}


\begin{thebibliography}{10}

\bibitem{Au}
T.~Aubin.
\'Equations diff\'erentielles non lin\'eaires et probl\`eme de Yamabe concernant la courbure scalaire.
\emph{J. Math. Pures Appl.}, 55(9):269--296, 1976.
 
\bibitem{CM2}
V.~Combet, Y. Martel.
Construction of multi-bubble solutions for the critical gKdV equation.
\emph{SIAM J. Math. Anal.}, 50(4) (2018), 3715--3790

\bibitem{CDM}
C.~Cortazar, M.~del Pino and M.~Musso.
Green's function and infinite-time bubbling in the critical nonlinear heat equation.
To appear in \emph{J. Eur. Math. Soc.} Preprint arXiv:1604.07117.

\bibitem{CMM}
R.~C\^ote, Y.~Martel and F.~Merle.
Construction of multi-soliton solutions for the $L^2$-supercritical gKdV and NLS equations.
\emph{Rev. Mat. Iberoam.} \textbf{27} (2011), no.~1, 273--302.

\bibitem{dPMW}
M.~del Pino, M.~Musso and J. Wei.
Type II blow-up in the 5-dimensional energy-critical heat equation.
\emph{Acta Math. Sin.} \textbf{35} (2019), no. 6, 1027--1042.
 
\bibitem{dPMW2} 
M. del Pino, M. Musso, J. Wei.
Infinite time blow-up for the 3-dimensional energy-critical heat equation.
Preprint arXiv:1705.01672.

\bibitem{dPMW3}
M. del Pino, M. Musso, J. Wei.
Geometry driven Type II higher dimensional blow-up for the critical heat equation.
Preprint arXiv:1710.11461.

 
\bibitem{DK}
R.~Donninger and J.~Krieger.
Nonscattering solutions and blowup at infinity for the critical wave equation.
\emph{Math. Ann.}, 357(1):89--163, 2013.

\bibitem{DKM1} T.~Duyckaerts, C.~E.~Kenig and F.~Merle.
Universality of blow-up profile for small radial type II blow-up solutions of the energy-critical wave equation.
\emph{J. Eur. Math. Soc.} \textbf{13} (2011), no. 3, 533--599.

\bibitem{DKM4} T.~Duyckaerts, C.~E.~Kenig and F.~Merle.
Classification of radial solutions of the focusing, energy-critical wave equation.
\emph{Cambridge Journal of Mathematics} \textbf{1} (2013), 75--144. 

\bibitem{DKM6} T.~Duyckaerts, C.~E.~Kenig and F.~Merle.
Solutions of the focusing nonradial critical wave equation with the compactness property.
\emph{Ann. Sc. Norm. Super. Pisa Cl. Sci.} (5) 15 (2016), 731--808.

\bibitem{DJKM1} T.~Duyckaerts, H.~Jia, C.~E.~Kenig and F.~Merle.
 Soliton resolution along a sequence of times for the focusing energy-critical wave equation.
\emph{Geom. Funct. Anal.} \textbf{27} (2017), no. 4, 798--862. 

\bibitem{PG98}
P.~G\'erard.
Description du d\'efaut de compacit\'e de l'injection de Sobolev.
\emph{ESAIM Control Optim. Calc. Var.} \textbf{3}: 213--233, 1998.

\bibitem{GSV}
J.~Ginibre, A.~Soffer and G.~Velo.
The global Cauchy problem for the critical nonlinear wave equation.
\emph{J. Funct. Anal.}, 110:96--130, 1992.

\bibitem{FHV} S. Filippas, M. A. Herrero, J. J. L. Velazquez.
Fast blow-up mechanisms for sign-changing solutions of a semilinear
parabolic equation.
\emph{R. Soc. Lond. Proc. Ser. A Math. Phys. Eng. Sci.} \textbf{456} no. 2004 (2000)
2957--2982.

\bibitem{Ha}
J. Harada,
A higher speed type II blowup for the five dimensional energy-critical heat equation.
Preprint arXiv 1906.03976.

\bibitem{HR} M.~Hillairet and P.~Rapha\"el.
Smooth type II blow up solutions to the four dimensional energy-critical wave equation.
\emph{Anal. PDE}, 5(4):777--829, 2012.

\bibitem{JJjfa} J.~Jendrej.
Construction of type II blow-up solutions for the energy-critical wave equation in dimension~5. 
\emph{J. Funct. Anal.} \textbf{272} (2017), no.~3, 866--917.

\bibitem{JJnon} J.~Jendrej.
Nonexistence of radial two-bubbles with opposite signs for the energy-critical wave equation.
\emph{Ann. Sc. Norm. Super. Pisa Cl. Sci.} (5) Vol. XVIII (2018), 1--44.

\bibitem{JJnls} J.~Jendrej.
Construction of two-bubble solutions for the energy-critical NLS.
\emph{Anal. PDE} \textbf{10} (2017) no. 8, pp. 1923--1959.

\bibitem{JJwave} J.~Jendrej.
Construction of two-bubble solutions for energy-critical wave equations.
\emph{Amer. J. Math.} \textbf{141} (2019), no. 1, 55--118.
 
\bibitem{JL} J.~Jendrej and A.~Lawrie. 
Two-bubble dynamics for threshold solutions to the wave maps equation.
\emph{Invent. Math.} \textbf{213} (2018), no. 3, 1249--1325.
 
\bibitem{KM} C.~E.~Kenig and F.~Merle.
Global well-posedness, scattering and blow-up for the energy-critical focusing non-linear wave equation.
\emph{Acta. Math.}, 201(2):147--212, 2008.

\bibitem{KS} J. Krieger and W. Schlag.
Full range of blow up exponents for the quintic wave equation in three dimensions.
\emph{J. Math Pures Appl.}, 101(6):873--900, 2014.

\bibitem{KSTmaps}
J. Krieger, W. Schlag, D. Tataru.
Renormalization and blow up for charge one equivariant critical wave maps. 
\emph{Invent. Math.} \textbf{171} (2008), no. 3, 543--615.

\bibitem{KST} J.~Krieger, W.~Schlag and D.~Tataru.
Slow blow-up solutions for the $H^1(\bR^3)$ critical focusing semilinear wave equation.
\emph{Duke Math. J.}, 147(1):1--53, 2009.

\bibitem{Ma1}
Y.~Martel.
Asymptotic $N$-soliton-like solutions of the subcritical and critical generalized Korteweg--de Vries equations.
\emph{Amer. J. Math.} \textbf{127} (2005), no.~5, 1103--1140.
 
\bibitem{MMR2}
Y.~Martel, F.~Merle and P.~Rapha\"el.
Blow up for the critical generalized Korteweg--de Vries equation. II:~Minimal mass dynamics.
\emph{J. Eur. Math. Soc.} \textbf{17} (2015), no.~8, 1855--1925.

\bibitem{MMwave1}
Y. Martel and F. Merle.
Construction of multi-solitons for the energy-critical wave equation in dimension 5. 
\emph{Arch. Ration. Mech. Anal.} \textbf{222} (2016), no. 3, 1113--1160.

\bibitem{MMwave2}
Y. Martel and F. Merle. 
Inelasticity of soliton collisions for the 5D energy-critical wave equation. 
\emph{Invent. Math.} \textbf{214} (2018), no. 3, 1267--1363.

\bibitem{MR}
Y. Martel and P. Raphaël.
Strongly interacting blow up bubbles for the mass critical NLS. 
\emph{Annales scientifiques de l’École normale supérieure}, \textbf{51}, fascicule 3 (2018), 701--737. 

\bibitem{Mcmp}
F.~Merle.
Construction of solutions with exactly $k$ blow-up points for the Schr\"odinger equation with critical nonlinearity.
\emph{Comm. Math. Phys.} \textbf{129} (1990), no.~2, 223--240.

\bibitem{Pi}
M. Pillai,
A continuum of infinite time blow-up solutions to the energy-critical wave maps equation.
Preprint arXiv:1905.00167

\bibitem{RS}
P.~Rapha\"el and J.~Szeftel.
Existence and uniqueness of minimal blow-up solutions to an inhomogeneous mass critical NLS.
\emph{J. Amer. Math. Soc.} \textbf{24} (2011), no.~2, 471--546.

\bibitem{Rey}
O.~Rey.
The role of the Green's function in a non-linear elliptic equation involving the critical Sobolev exponent.
\emph{J. Funct. Anal.} \textbf{89} (1990), no.~1, 1--52.


\bibitem{Ro}
C. Rodriguez.
Profiles for the radial focusing energy-critical wave equation in odd dimensions.
\emph{Adv. Differential Equations} \textbf{21} (2016), no. 5--6, 505--570.

\bibitem{Sch}
R. Schweyer.
Type II blow-up for the four dimensional energy-critical semi linear heat equation.
\emph{J. Funct. Anal.} \textbf{263} no. 12 (2012) 3922--3983.

\bibitem{SS}
J.~Shatah and M.~Struwe.
Well-posedness in the energy space for semilinear wave equations with critical growth.
\emph{Internat. Math. Res. Notices}, 7:303--309, 1994.

\bibitem{SSbook}
J.~Shatah and M.~Struwe.
\emph{Geometric Wave Equations}, volume 2 of \emph{Courant Lecture Notes in Mathematics}. AMS, 2000.

\bibitem{Ta}
G.~Talenti.
Best constant in Sobolev inequality.
\emph{Ann. Mat. Pura Appl.}, 110(4):353--372, 1976.
 
\end{thebibliography}
\end{document}